\documentclass[aop,MSNbibl,dvips]{arximspdf}
\usepackage{graphicx}
\doi{10.1214/12-AOP814} %kopijuoti is PTS
\volume{41}
\issue{4}
\pubyear{2013}
\firstpage{2426}
\lastpage{2478}

\makeatletter

\mathchardef\varXi="0104
\mathchardef\varTheta="0102

\newcommand{\rrvert}{\vert}
\newcommand{\llvert}{\vert}
\newcommand{\eqref}[1]{(\ref{#1})}
\newtheorem{theorem}{Theorem}[section]
\newtheorem{lemma}[theorem]{Lemma}
\newtheorem{definition}[theorem]{Definition}
\newtheorem{corollary}[theorem]{Corollary}
\newtheorem{proposition}[theorem]{Proposition}
\newproclaim{remark}[theorem]{Remark}
\newproclaim{example}[theorem]{Example}

\newcommand{\R}{\mathbb{R}}

\newcommand{\N}{\mathbb{N}}
\newcommand{\Z}{\mathbb{Z}}
\newcommand{\Q}{Q}
\renewcommand{\P}{\mathbb{P}}
\renewcommand{\d}{d}

\newcommand{\W}{\mathbb{W}}
\newcommand{\EE}{\mathbb{E}}

\newcommand{\leb}{\frak{L}}
\newcommand{\CCost}{\mathsf{Cost}}
\newcommand{\Cost}{\mathfrak{Cost}}
\newcommand{\CCo}{\mathsf{C}}
\newcommand{\supp}{\operatorname{supp}}

\newcommand{\abs}[1]{\llvert #1\rrvert}

\makeatother

\begin{document}
\begin{frontmatter}

\title{Optimal transport from Lebesgue to Poisson}
\runtitle{Optimal transport from Lebesgue to Poisson}

\begin{aug}
\author[A]{\fnms{Martin} \snm{Huesmann}\ead[label=e1]{huesmann@iam.uni-bonn.de}}
\and
\author[A]{\fnms{Karl-Theodor} \snm{Sturm}\corref{}\ead[label=e2]{sturm@iam.uni-bonn.de}}
\runauthor{M. Huesmann and K.-T. Sturm}
\affiliation{University of Bonn}
\address[A]{University of Bonn\\
Endenicher Allee 60\\
53115 Bonn\\
Germany\\
\printead{e1}\\
\phantom{E-mail:\ }\printead*{e2}} %adresu isvedimo komanda gale!
\end{aug}

% HISTORY:
\received{\smonth{4} \syear{2011}}
\revised{\smonth{8} \syear{2012}}

% ABSTRACT
%
\begin{abstract}
This paper is devoted to the study of couplings of the Lebesgue measure
and the Poisson point process.
We prove existence and uniqueness of an optimal coupling whenever the
asymptotic mean transportation cost is finite. Moreover, we give
precise conditions for the latter which demonstrate a sharp threshold
at $d=2$.
The cost will be defined in terms of an arbitrary increasing function
of the distance.

The coupling will be realized by means of a transport map (``allocation
map'') which assigns to each Poisson point a set (``cell'') of Lebesgue
measure~1. In the case of quadratic costs, all these cells will be
convex polytopes.
\end{abstract}

% KEYWORDS


\end{frontmatter}

%s1 #&#
\section{Introduction and statement of main results}\label{sec1}
(a)  The theory of \emph{optimal transportation} studies
couplings between two probability measures $\lambda$ and $\nu$ on
$\mathbb{R}^d$ which minimize the total transportation cost. A coupling
is interpreted as a plan how to transport $\lambda$ into $\nu$.
Transporting a unit of mass from $a$ to $b$ produces cost of amount
$c(a,b)$, where $c(\cdot,\cdot)$ is a given cost function. Of
particular interest are couplings $q$ which are induced by transport
maps, that is, $q=(\mathrm{id},\psi)_*\lambda$ for some map $\psi\dvtx\R^d\to\R^d$
with $\psi_*\lambda=\nu.$

A \emph{fair allocation} for a simple point process in $\R^d$ is a
coupling of the Lebesgue measure $\leb$ and the point process $\mu^\bullet$ induced by a transport map, that is, there is a map $\Psi
\dvtx\Omega\times\R^d\to\R^d$ such that for $\P$-almost every
$\omega\in
\Omega$ the map $\Psi^\omega\dvtx\R^d\to\R^d$ transports the Lebesgue
measure into the point process: $\Psi^\omega_*\leb=\mu^\omega.$
Such an
allocation is called \emph{factor allocation} if it is a measurable
function of the point process (i.e., it measurably depends only on the
given point process).

In this article we connect these two theories by constructing fair
allocations between the Lebesgue measure and point processes using
tools from optimal transportation. Instead of considering the total
transportation cost we ask for minimizers of the \emph{cost per unit
mass}. Good estimates on the transportation cost will directly imply
good tail estimates for the distribution of the transport distance.

Moreover, the techniques developed in this article allow us to
construct a \emph{fair factor allocation} with the best possible tail
estimate and also to derive new estimates on the transportation cost
between the Lebesgue measure and a Poisson point process.

We now describe our results in more detail.

(b) A point process $\mu^\bullet\dvtx\Omega\to\mathcal N(\R^d)$ is
a random variable with values in the space of integer valued Radon
measure. Put $\Xi(\omega)=\supp(\mu^\omega).$ Then, $\mu^\bullet
$ has
the representation $\mu^\bullet\dvtx\omega\mapsto\mu^\omega=\sum_{\xi\in\Xi
(\omega)}k(\xi)\cdot\delta_\xi$ with $k(\xi)\in\N.$ $\mu^\bullet$ is
called equivariant if for all Borel sets $A\in\mathcal B(\R^d)$ we have
$\mu^{\omega+z}(A+z)=\mu^\omega(A).$ Here, we interprete $\omega
+z$ as
the support of $\mu^\omega$ translated by $z$; see Section~\ref{spp}.

Given an equivariant point process $\mu^\bullet\dvtx \omega\mapsto
\mu^\omega=\sum_{\xi\in\Xi(\omega)}k(\xi)\cdot\delta_\xi$ on
$\R^d$ with
unit intensity, we consider the set $\Pi$ of all \emph{couplings}
$q^\bullet$ of the Lebesgue measure $\leb$ and the point process---that
is, the set of measure-valued random variables $\omega\mapsto q^\omega$
s.t. for a.e. $\omega$ the measure $q^\omega$ on $\R^d\times\R^d$
is a
coupling of $\leb$ and $\mu^\omega$---and we ask for a minimizer of the
\emph{asymptotic mean cost functional}
\[
{\mathfrak C}_\infty \bigl(q^\bullet \bigr):=
\liminf_{n\to\infty} \frac1{\leb(B_n)}\EE \biggl[\int
_{\R^d\times B_n}\vartheta\bigl(|x-y|\bigr) \,dq^\bullet(x,y) \biggr].
\]
Here $B_n:=[0,2^n)^d\subset\R^d$. The \emph{scale} $\vartheta\dvtx
\R_+\to\R_+$ will always be some strictly increasing, continuous
function with
$\vartheta(0)=0$ and $\lim_{r\to\infty}\vartheta(r)=\infty.$

A coupling $\omega\mapsto q^\omega$ of the Lebesgue measure and the
point process is called \emph{optimal} if it minimizes the asymptotic
mean cost functional and if it is \emph{equivariant} in the sense that
$q^{\omega+z}(A+z,B+z)=q^\omega(A,B)$ for all $z\in\R^d$ and Borel sets
$A,B\in\mathcal B(\R^d).$
Our main result states the following:
%
%
%th1.1 #&#
\begin{theorem}\label{mainthm1}
If the asymptotic mean transportation cost
%
%
%e1 #&#
\begin{equation}
{\mathfrak c}_\infty:= \liminf_{n\to\infty} \inf_{q^\bullet\in
\Pi
}
\frac1{\leb(B_n)}\EE \biggl[\int_{\R^d\times B_n}
\vartheta\bigl(|x-y|\bigr) \,dq^\bullet(x,y) \biggr] \label{liminfinf}
\end{equation}
is finite, then there exists a unique optimal coupling of the Lebesgue
measure and the point process $\mu^\bullet$.
\end{theorem}

(c) The unique optimal coupling $q^\omega$ can be represented
as $(\mathrm{id},T^\omega)_*\leb$ for some map $T^\omega\dvtx\R^d\to
\mathrm{supp}(\mu^\omega)\subset\R^d$ measurably only dependent on
the sigma algebra
generated by the point process.
In other words, $T^\omega$ defines a \emph{fair factor allocation}. Its
inverse map assigns to each point $\xi$ of the point process
(``center'') a set (``cell'') of Lebesgue measure $\mu^\omega(\xi
)\in\N$.
If the point process is simple, then all these cells have volume 1.
In the case of quadratic cost, that is, $\vartheta(r)=r^2$, the cells
will be convex polytopes.
The transport map will be given as
$T^\omega=\nabla\varphi^\omega$ for some convex function $\varphi^\omega\dvtx\R^d\to\R$ and induces a Laguerre tessellation; see
\cite
{lautensack2007}.

In the case $\vartheta(r)=r$ the transportation map induces a
Johnson--Mehl diagram; see~\cite{aurenhammer1991}. For the many results
on and applications of these tessellations see the references in \cite
{lautensack2007} and~\cite{aurenhammer1991}. In the light of these
results one might interpret the optimal coupling as a generalized tessellation.

(d)
As a particular corollary to Theorem~\ref{mainthm1} we conclude that ${\mathfrak
c}_\infty=  \inf_{q^\bullet\in\Pi}{\mathfrak C}_\infty(q^\bullet)$ and
that the infimum is always attained; more precisely, it is attained by
an equivariant coupling $q^\bullet$.
For equivariant couplings $q^\bullet$ the mean cost functional
$ \frac1{\leb(A)}\EE[\int_{\R^d\times A}\vartheta(|x-y|)
\,dq^\bullet
(x,y) ]$, however,
is independent of $A\subset\mathbb R^d$. Hence,
\[
{\mathfrak c}_\infty= \inf_{q^\bullet\in\Pi_{\mathrm{eqv}}} \EE \biggl[\int
_{\R^d\times[0,1)^d}\vartheta\bigl(|x-y|\bigr) \,dq^\bullet(x,y) \biggr],
\]
where
$\Pi_{\mathrm{eqv}}$ now denotes the set of all equivariant couplings of the
Lebesgue measure and the point process.

Moreover, for equivariant couplings, $\EE[\vartheta(|x-T^\bullet
(x)|) ]$ the mean cost of transportation
of a Lebesgue point $x$ to the center of its cell is independent of
$x\in\mathbb R^d$. Hence,
%
%
%e2 #&#
\begin{equation}
\mathfrak c_\infty=\inf_{T^\bullet}\EE \bigl[\vartheta
\bigl(\bigl|0-T^\bullet(0)\bigr| \bigr) \bigr], \label{mapinf}
\end{equation}
where the infimum is taken over all equivariant maps $T\dvtx\R^d\times\Omega
\to\R^d$ with
${T^\omega}_*\leb=\mu^\omega$ for a.e. $\omega$. And again: the infimum
is attained by a unique such~$T$. Let us point out that identity (\ref
{mapinf}) allows us to resolve the asymmetry in the integration domain
in equation (\ref{liminfinf}): we equally well may replace the domain
of integration $\R^d\times B_n$ by $B_n\times\R^d$.

(e)
Analogous results will be obtained in the more general case of optimal
``semicouplings'' between the Lebesgue measure and point processes of
``subunit'' intensity.

We develop the theory of optimal semicouplings as a concept of
independent interest. Optimal semicouplings are solutions of a twofold
optimization problem: the optimal choice of a density $\rho\le1$ of the
first marginal $\mu_1$ and subsequently the optimal choice of a
coupling between $\rho\mu_1$ and $\mu_2$.
This twofold optimization problem can also be interpreted as a
transport problem with free boundary values; see Figure~\ref{fig1}.

%
%f1 #&#
\begin{figure}

\includegraphics{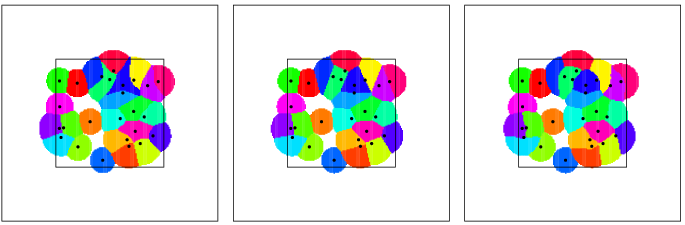}

\caption{Optimal semicoupling of Lebesgue and 25 points in the cube
with cost function $c(x,y)=|x-y|^p$ and (from left to right) $p=1, 2,
4$,
respectively.}\label{fig1}
\end{figure}

Given a point process of subunit intensity and finite mean
transportation cost, we prove that there exists a unique optimal
semicoupling between the Lebesgue measure and the point process. It can
be represented on $\R^d\times\R^d$ as before as $q^\omega
=(\mathrm{id},T^\omega
)_*\leb$ in terms of a transport map
$T^\omega\dvtx\R^d\to\supp[\mu^\omega]\cup\{\eth\}$
where $\eth$ now denotes an isolated point (``cemetery'') added to $\R^d$.

(f)
In any case, we prove that the unique transport map $T^\omega$ can be
obtained as the limit of a suitable sequence of transport maps which
solve the optimal transportation problem between the Lebesgue measure
and the point process restricted to bounded sets.

More precisely, for $z\in\Z^d$ and $\gamma\in\Gamma:=(\{0,1\}^d)^\N$
consider the ``doubling sequence'' of cubes
\[
B_n(z,\gamma) = z-\sum_{k=1}^n2^{k-1}
\gamma_k+\bigl[0,2^n\bigr)^d.
\]
Note that the cube $B_n(z,\gamma)$ is one of the subcubes obtained by
subdividing $B_{n+1}(z,\gamma)$ into $2^d$ cubes of half edge length.
Let $T_{z,n}(\cdot,\omega,\gamma)\dvtx\R^d\to\supp[\mu^\omega]\cup
\{\eth\}$ be
the transport map for the unique optimal semicoupling between $\leb$
and $1_{B_n(z,\gamma)}\cdot\mu^\omega$, that is, for the optimal
transport of an optimal ``submeasure'' $\rho^\omega\cdot\leb$ to the
point process restricted to the cube $B_n(z,\gamma)$.

%
%th1.2 #&#
\begin{theorem}\label{mainthm2}
For every $z\in Z^d$ and every bounded Borel set $M\subset\R^d$,
\[
\lim_{n\to\infty} (\leb\otimes\P\otimes\nu) \bigl( \bigl\{ (x,\omega,\gamma)
\in M\times\Omega\times\Gamma\dvtx T_{z,n}(x,\omega,\gamma) \not = T(x,
\omega) \bigr\} \bigr) = 0,
\]
where $\nu$ denotes the Bernoulli measure on $\Gamma$.
\end{theorem}

(g) If $\mu^\bullet$ is a Poisson point process with intensity
$\beta\le1$ we have rather sharp estimates for the asymptotic mean
transportation cost to be finite.

%
%th1.3 #&#
\begin{theorem}\label{costthm}
\textup{(i)} Assume $d\ge3$ (and $\beta\le1$) or $\beta<1$ (and $d\ge1$). Then
there exists a constant $0<\kappa<\infty$ s.t.
\[
\limsup_{r\to\infty}\frac{\log\vartheta(r)}{r^d}<\kappa \quad\Longrightarrow\quad{\mathfrak
c}_\infty<\infty\quad\Longrightarrow\quad\liminf_{r\to\infty}\frac{\log
\vartheta(r)}{r^d}
\le\kappa.
\]

\textup{(ii)} Assume $d\le2$ and $\beta=1$. Then for any concave $\hat
\vartheta\dvtx[1,\infty)\to\R$ dominating~$\vartheta$
\[
\int_1^\infty\frac{\hat\vartheta(r)}{r^{1+d/2}}\, \d r<\infty
\quad\Longrightarrow\quad{\mathfrak c}_\infty<\infty\quad\Longrightarrow\quad
\liminf_{r\to\infty}\frac{\vartheta(r)}{r^{d/2}}=0.
\]
\end{theorem}

The first implication in assertion (ii) is new.
Assertion (i) in the case $\beta=1$ is due to Holroyd and Peres \cite
{extra-heads}, based on a fundamental result of Talagrand~\cite{Talagrand94}.
The first implication in assertion (i) in the case $\beta<1$ was proven
by Hoffman, Holroyd and Peres~\cite{stable-marriage}.
The second implication in assertion (ii) is due to~\cite{holroyd2001find}.

Now let us consider the particular case of $L^p$ transportation cost,
that is, \mbox{$\vartheta(r)=r^p$}.
%
%
%co1.4 #&#
\begin{corollary} \textup{(i)} For all $d\in\N$, all $\beta\le1$ and $p\in
(0,\infty)$ the asymptotic
mean $L^p$-transportation cost ${\mathfrak c}_\infty$ is finite if and
only if
\[
p<\overline p:=\cases{ %
\infty,&\quad
$\mbox{for } d\ge3 \mbox{ or }\beta<1;$
\vspace*{2pt}\cr
\displaystyle\frac d2,&\quad$\mbox{for } d\le2 \mbox{ and }\beta=1.$}
\]

\textup{(ii)} If $\beta=1$, then for all $p\in(0,\infty)$ there exist
constants $0< k\le k'<\infty$ s.t.
for all $d>2(p\wedge1)$
\[
k\cdot d^{p/2} \le\frak c_\infty\le k'\cdot
d^{p/2}.
\]
\end{corollary}

(h)
The study of fair allocations for point processes is an important and
hot topic of current research; see, for example, \cite
{extra-heads,timar2008invariant,matching09} and references therein. A landmark
contribution was the construction of the \emph{stable marriage} between
Lebesgue measure and an ergodic translation invariant simple point
process~\cite{stable-marriage}. One of the challenges is to produce
allocations with fast decay of the distance of a typical point in a
cell to its center or of the diameter of the cell. The \emph
{gravitational allocation}~\cite{gravity,phasegravity} in $d\geq3$ was
the first allocation with exponential decay. Moreover, all the cells
are connected and contain their center. However, the decay was not yet
as good as the decay of a \emph{random allocation} constructed in
\cite
{extra-heads}.

On the other hand, during the last decade the theory of optimal
transportation (see, e.g.,~\cite{Rachev-Ruesch,Villani1})
has attracted lot of interest and has produced an enormous amount of
deep results,
striking applications and stimulating new developments, among others in
PDEs (e.g.,~\cite{Brenier,Otto01,AGS}), evolution
semigroups (e.g.,~\cite{Otto-Villani,Ambrosio-Savare-Zambotti,Ohta-Sturm}) and geometry (e.g.,~\cite{Sturm-Acta1,Sturm-Acta2,Lott-Villani,villani2009optimal,Ohta-Finsler}).
Ajtai, Koml\'os and Tusn\'ady as well as Talagrand and others studied
the problem of matchings and allocation of independently distributed
points in the unit cube in terms of transportation cost (\cite
{Ajtai-K-T,Talagrand94} and references therein). For further
studies of invariant transports between random measures in more general
spaces we refer to~\cite{last2008invariant}\footnote{In the course of
the refereeing process of this paper a construction of a fair
allocation for the Poisson point process with optimal tail behavior of
the diameter of a typical cell was presented by Mark\'o and Timar \cite
{marko2011poisson} using the algorithm of Ajtai, Koml\'os and Tusn\'ady.}.

(i)
In all the optimal transportation problems considered in the
aforementioned contributions, however, the marginals have finite total mass.\vadjust{\goodbreak}
Our paper seems to be the first to prove existence and uniqueness of a
solution to an optimal transportation problems for which the total
transportation cost is infinite.

More precisely, the main contributions of the current paper are:
\begin{itemize}
\item We present a concept of ``optimality'' for (semi-) couplings
between the Lebesgue measure and a point process.
\item We prove existence and uniqueness of an optimal semicoupling
whenever there exists a semicoupling with finite asymptotic mean
transportation cost.
\item We prove that for a.e. doubling sequence of boxes $
(B_n(z,\gamma) )_{n\in\N}$ the sequence of optimal semicouplings
$q^\bullet_{n,z,\gamma}$ between the Lebesgue measure and the point
process restricted to the box $B_n(z,\gamma)$ will converge. More
precisely, the sequence $q^\bullet_{n,z,\gamma}$ will converge as
$n\to
\infty$ toward a unique optimal semicoupling $q^\bullet$ between the
Lebesgue measure and the point process.
\item We prove that the asymptotic mean transportation cost for the
Poisson point process in $d\le2$ is finite for $L^p$-costs with $p<d/2$
and also for more general scale functions like $\vartheta
(r)=r^{d/2}\cdot\frac1{(\log r)^\alpha}$ with $\alpha>1$.
\end{itemize}

%s1.1 #&#
\subsection{Outline}\label{sec1.1}
The article is divided into five parts. The core material with the
proofs of the main theorems is contained in Sections~\ref{su} to \ref
{costsection}. These three sections are rather independent of each other.

In Section~\ref{sset-up} we start by recalling the relevant
definitions and objects we work with. We also state an importation
technical result, Theorem~\ref{euQ+q}, the existence and uniqueness
result of optimal semicouplings on bounded sets. The proof of this
theorem is deferred to Section~\ref{sapp} because it is a purely
deterministic result on transportation problems between \emph{finite}
measures whereas the rest of the article deals with transportation
problems between random measures with \emph{infinite} mass. The key
idea for the proof is to show that every minimizer has to be
concentrated on a certain graph. Then, existence can be shown via lower
semicontinuity plus compactness. Uniqueness follows from the
observation that a convex combination of optimal semicouplings can only
be concentrated on a graph if all optimal semicouplings are
concentrated on the same graph.

In Section~\ref{su} we proof the uniqueness part of Theorem \ref
{mainthm1}. The idea for the proof is again to show that every optimal
semicoupling has to be concentrated on the graph of some function. To
this end, we introduce the concept of local optimality. A semicoupling
$q^\bullet$ is called locally optimal if and only if for $\P$-almost
all~$\omega$ the restriction of $q^\omega$ to any bounded Borel set $A,
1_{\R^d\times A}q^\omega$ is optimal between its marginals in the
classical sense. Using equivariance, we show that every optimal
semicoupling is locally optimal. Hence, by applying Theorem \ref
{euQ+q} we get the existence of a transportation map and therefore uniqueness.

The proof of the existence part of Theorem~\ref{mainthm1} is
presented in the first part of Section~\ref{sc}. The idea is to
approximate the optimal semicoupling by solutions to classical optimal
transportation problems on bounded regions. The main problem to
overcome is to control the contribution of a small fixed observation
window to the total asymptotic mean transportation cost. The solution
is not to consider a deterministic exhausting sequence of cubes, but a
random sequence of cubes. This second randomization causes a
symmetrization and induces tightness of this sequence. It could also be
seen as a way to enforce the equivariance of the limiting measure. The
uniqueness of optimal semicouplings then allows us to remove the second
randomization again and also to deduce ``quenched'' results in the
second part of Section~\ref{sc} which finally proves Theorem~\ref{mainthm2}.

In Section~\ref{costsection}, we prove Theorem~\ref{costthm}. The
estimates are based on an explicit construction of a semicoupling
between $\leb$ and $1_{[0,2^n)^d}\mu^\bullet.$ The transportation cost
estimate can thereby be reduced to the estimates of moments, central
moments and inverse moments of Poisson random variables. The advantage
of this approach is that it allows us to get fairly reasonable
estimates of constants and, more importantly, it is also potentially
applicable to other cases of interest.

%s2 #&#
\section{Set-up and basic concepts}\label{sset-up}

$\leb$ will always denote the Lebesgue measure on $\R^d$. The
complement of a set $A\subset\R^d$ will be denoted by $\complement A$.
The push forward of a measure $\rho$ by a map $S$ will be denoted by
$S_*\rho$.

%s2.1 #&#
\subsection{Couplings and semicouplings}
For each Polish space $X$ (i.e., separable, complete metrizable space)
the set of measures on $X$---equipped with its Borel $\sigma
$-field---will be denoted by $\mathcal M(X)$.
Given any ordered pair of Polish spaces $X,Y$ and measures $\lambda\in
\mathcal M(X), \mu\in\mathcal M(Y)$, we say that a measure $q\in
\mathcal M(X\times Y)$ is a \emph{semicoupling} of $\lambda$ and $\mu$,
briefly $q\in\Pi_{s}(\lambda,\mu)$,
if and only if
the (first and second, resp.) marginals satisfy
\[
(\pi_1)_\ast q\leq\lambda, \qquad (\pi_2)_{\ast}q=
\mu,
\]
that is, if and only if
$q(A\times Y)\le\lambda(A)$ and $q(X\times B)=\mu(B)$ for all Borel
sets $A\subset X, B\subset Y$.
The semicoupling $q$ is called \emph{coupling}, briefly $q\in\Pi
(\lambda,\mu)$, if and only if, in addition,
\[
(\pi_1)_\ast q= \lambda.
\]

Existence of a coupling requires that the measures $\lambda$ and $\mu$
have the same total mass.
If the total masses of $\lambda$ and $\mu$ are finite and equal, then
the ``renormalized'' product measure
$q=\frac1{\lambda(X)}\lambda\otimes\mu$
is always a coupling of $\lambda$ and $\mu$.

If $\lambda$ and $\mu$ are $\Sigma$-finite, that is, $\lambda=\sum_{n=1}^\infty\lambda_n$, $\mu=\sum_{n=1}^\infty\mu_n$ with finite
measures $\lambda_n\in\mathcal M(X)$, $\mu_n\in\mathcal M(Y)$---which
is the case for all Radon measures---and if both\vadjust{\goodbreak} of them have infinite
total mass, then there always exists a $\Sigma$-finite coupling of them.
[Indeed, then the $\lambda_n$ and $\mu_n$ can be chosen to have unit
mass and $q=\sum_n (\lambda_n\otimes\mu_n)$ does the job.]

See also~\cite{figalli2010optimal} for the related concept of \emph
{partial coupling}.

%s2.2 #&#
\subsection{Point processes}\label{spp}

Throughout this paper, $\mu^\bullet$ will denote an equivariant point
process of subunit intensity, modeled on some probability space
$(\Omega
, \mathfrak{A}, \mathbb P)$. For convenience, we will assume that
$\Omega$ is a compact separable metric space and $\mathfrak{A}$ its
completed Borel field. These technical assumptions are only made to
simplify the presentation.

Recall that a \emph{point process} is a measurable map $\mu^\bullet
\dvtx\Omega\to\mathcal M(\mathbb R^d)$, $\omega\mapsto\mu^{\omega
}$ with
values in the subset
$\mathcal N(\R^d)$ of locally finite \emph{counting measures} on $\R^d$.
It is a particular example of a random measure, characterized by the
fact that
$\mu^\omega(A)\in\N_0$
for $\P$-a.e. $\omega$ and every bounded Borel set $A\subset\R^d$.
It can always be written as
\[
\mu^\omega=\sum_{\xi\in\Xi(\omega)}k(\xi)
\delta_\xi
\]
with some countable set $\Xi(\omega)\subset\R^d$ without accumulation
points and with numbers $k(\xi)\in\N$.
The point process is called \emph{simple} if and only if $k(\xi)=1$ for
all $\xi\in\Xi(\omega)$ and a.e. $\omega$ or, in other words, if and
only if
$\mu(\{x\})\in\{0,1\}$
for every $x\in\R^d$ and a.e. $\omega$.

We assume that the probability space $(\Omega, \mathfrak A, \P)$ admits
a measurable flow $\theta\dvtx\R^d\times\Omega\to\Omega$ such
that the
point process $\mu^\bullet$ is \emph{$\R^d$-equivariant} or just
equivariant, that is,
\[
\mu^{\theta_z(\omega)}(A+z) = \mu^\omega(A)
\]
for all Borel sets $A\in\mathcal B(\R^d).$ Moreover, we assume that
$\P
$ is stationary, that is, invariant under the flow
\[
\P\circ\theta= \P.
\]
In particular, this implies that $\mu^\bullet$ is \emph{translation
invariant} in the usual sense, that is,
\[
(\tau_z)_*\mu^\bullet\stackrel{\mathrm{(d)}}= \mu^\bullet
\]
for each $z\in\R^d$. We interpret the flow as a shift of the support of
$\mu^\bullet$ and therefore write $\theta_z(\omega)=\omega+z$; see
also Example 2.1 of~\cite{last2008invariant}.

To split the translation invariance into equivariance and stationarity
has the huge advantage that equivariance is stable under addition
whereas translation invariance is not. It is not really a restriction
as we can always take the canonical realization as a probability space;
again see Example 2.1 of~\cite{last2008invariant}.

We say that $\mu^\bullet$ has \emph{subunit intensity} if and only if
$\EE[\mu^\bullet(A) ]\le\leb(A)$ for all Borel sets $A\subset
\R^d$.
If ``$=$'' holds instead of ``$\le$'' we say that $\mu^\bullet$ has
\emph{unit intensity}.
A translation invariant\vadjust{\goodbreak} point process has subunit (or unit) intensity
if and only if its intensity
\[
\beta=\EE\bigl[\mu^\bullet \bigl( [0,1 )^d\bigr) \bigr]
\]
is $\le1$ (or $=1$, resp.).

Given a point process $\mu^\bullet$, the measure $d(\mu^\bullet
\mathbb
P)(y,\omega):=d\mu^\omega(y) \,d\mathbb P(\omega)$ on $\R^d\times
\Omega$
is called \emph{Campbell measure} of the random measure $\mu^\bullet$.

The most important example of an equivariant simple point process is
the \emph{Poisson point process} or \emph{Poisson random measure} with
intensity $\beta\le1$. It is characterized by:
\begin{itemize}
\item for each Borel set $A\subset\mathbb R^d$ of finite volume the
random variable $\omega\mapsto\mu^{\omega}(A)$ is Poisson distributed
with parameter $\beta\cdot\leb(A)$, and
\item for disjoint Borel sets $A_1,\ldots,A_k\subset\R^d$ the family
of random variables $\mu^{\omega}(A_1),\ldots,\mu^{\omega}(A_k)$
is independent.
\end{itemize}

There are some instances in which we need additional assumptions on
$\mu^\bullet$ (e.g., ergodicity, unit intensity). In each of these
cases we
will clearly point out the specific assumptions we make.

%s2.3 #&#
\subsection{Couplings of Lebesgue measure and the point process}
A (semi-) coupling of the Lebesgue measure $\leb\in\mathcal M(\R^d)$
and the point process $\mu^\bullet\dvtx \Omega\to\mathcal M(\R^d)$ is a
measurable map $q^\bullet\dvtx \Omega\to\mathcal M(\R^d\times\R^d)$ s.t.
for $\mathbb P$-a.e. $\omega\in\Omega$
\[
q^\omega\mbox{ is a (semi-) coupling of $\leb$ and $
\mu^\omega$}.
\]
We say that a measure $\Q\in\mathcal M(\mathbb R^d\times\mathbb
R^d\times\Omega)$ is an \emph{universal \textup{(}semi-\textup{)} coupling} of the
Lebesgue measure and the point process if and only if $d\Q(x,y,\omega)$
is a (semi-) coupling of the Lebesgue measure $d\leb(x)$ and of the
Campbell measure $d(\mu^\bullet\mathbb P)(y,\omega)$.

Disintegration of a universal (semi-) coupling w.r.t. the third
marginal yields a measurable map
$q^\bullet\dvtx \Omega\to\mathcal M(\R^d\times\R^d)$ which is a
(semi-) coupling of the Lebesgue measure $\leb$ and the point process
$\mu^\bullet$.
Conversely, given any (semi-) coupling $q^\bullet$ of the Lebesgue
measure $\leb$ and the point process $\mu^\bullet$, then its
Campbell measure
\[
\d\Q(x,y,\omega):= \d q^\omega(x,y)\,\d\P(\omega)
\]
defines a universal (semi-) coupling.

According to this one-to-one correspondence between $q^\bullet$
[(semi-) coupling of~$\leb$ and $\mu^\bullet$] and $Q=q^\bullet
\mathbb
P$ [(semi-) coupling of $\leb$ and $\mu^\bullet\mathbb P$], we will
freely switch between them. In many cases, the specification
``universal'' for \mbox{(semi-)} couplings of $\leb$ and $\mu^\bullet
\mathbb P$
will be suppressed. And quite often, we will simply speak of \emph
{\textup{(}semi-\textup{)} couplings of $\leb$ and $\mu^\bullet$}.

%s2.4 #&#
\subsection{Fair allocations}
Let $\mu^\bullet\in\mathcal N(\R^d)$ be given. A \emph{fair
allocation} of Lebes\-gue measure $\mathcal L$ to $\mu^\bullet$ is a
measurable map $\Psi^\bullet\dvtx\Omega\times\mathbb R^d\to
\mathbb R^d$,
$(\omega,x)\mapsto\Psi^\omega(x)$ such that for $\mathbb{P}$-almost
every $\omega$:\vadjust{\goodbreak} %
\begin{longlist}[(ii)]
\item[(i)] $\mathcal L (\mathbb R^d\setminus\bigcup_{\xi\in{\Xi
_\omega}} \Psi_\omega^{-1}(\xi) ) = 0$;
\item[(ii)] $\mathcal L (\Psi_\omega^{-1}(\xi) ) = 1$ for
all $\xi\in\Xi(\omega)$.
\end{longlist}
We call each \emph{configuration point} $\xi\in\Xi(\omega)$ a
\emph
{center}, and the set $(\Psi^\omega)^{-1}(\xi)$ the \emph{cell}
associated to the center $\xi.$ The allocation $\Psi^\bullet$ is called
equivariant if and only if $\Psi_{\omega}(x)=y \Rightarrow\forall
z\in\mathbb R^d\dvtx\Psi_{\theta_z\omega}(x+z)=y+z$. An allocation is
called \emph{factor allocation} if the random map $\omega\mapsto\Psi^\omega$
is measurable with respect to the $\sigma$-algebra generated
by $\mu^\bullet.$ For some examples on allocations and their connection
to Palm measures we refer to \cite
{extra-heads,stable-marriage,gravity} and references therein.

In particular, any allocation $\Psi^\bullet$ for $\mu^\bullet$ induces
a coupling $q^\bullet$ between $\leb$ and $\mu^\bullet$ via
$q^\bullet
=(\mathrm{id}, \Psi^\bullet)_*\leb.$

%s2.5 #&#
\subsection{The optimal transportation problem}\label{sot}
Given two probability measures~$\lambda$, $\mu$ on $\R^d$ and a
measurable cost function $c\dvtx\R^d\times\R^d\to\R$, the optimal
transportation problem between $\lambda$ and $\mu$ is to find a
minimizer of
\[
\int_{\R^d\times\R^d} c(x,y) \,dq(x,y)
\]
among all couplings $q$ of $\lambda$ and $\mu.$ A minimizer is called
\emph{optimal coupling}. Optimal couplings have many nice properties.
The most basic and also very intuitive one is that they are
concentrated on \emph{$c$-cyclical monotone} sets. A~set $N\subset\R^d\times\R^d$
is called $c$-cyclical monotone if and only if for all
$n\in\N$ and $(x_i,y_i)\in N$ for $i=1,\ldots,n$, we have
%
%
%e3 #&#
\begin{equation}
\label{cm} \sum_{i=1}^n
c(x_i,y_i) \leq\sum_{i=1}^n
c(x_i,y_{i+1}),
\end{equation}
where $y_{n+1}=y_1.$ The interpretation of cyclical monotonicity is
clear. If a coupling is optimal we cannot improve it, produce a
coupling with less cost, by breaking up and recoupling finitely many
coupled pairs of points. In fact, if the cost function is sufficiently
nice (continuous is much more than needed, see~\cite{betterplans}) also
the reverse direction holds. Any measure that is concentrated on a
$c$-cyclical monotone set is optimal. In many situations, the optimal
coupling is induced by a transportation map $T$, that is,
$q=(\mathrm{id},T)_*\lambda$. Then $T$ is $c$-cyclically monotone if and only if
its graph is $c$-cyclical monotone set. For more details on optimal
transportation and its many applications we refer to \cite
{Villani1,villani2009optimal,Rachev-Ruesch}.

%s2.6 #&#
\subsection{Cost functionals}

Throughout this paper,
$\vartheta$ will be a strictly increasing, continuous function from
$\mathbb R_+$ to $\mathbb R_+$ with $\vartheta(0)=0$ and $\lim_{r\to
\infty}\vartheta(r)=\infty$.
Given a \emph{scale function} $\vartheta$ as above we define the
\emph
{cost function}
\[
c(x,y)=\vartheta\bigl(|x-y| \bigr)
\]
on $\mathbb R^d\times\mathbb R^d$, the \emph{cost functional}
\[
\CCost(q)=\int_{\R^d\times\R^d}c(x,y) \,dq(x,y)\vadjust{\goodbreak}
\]
on $\mathcal M(\R^d\times\R^d)$
and the \emph{mean cost functional}
\[
\Cost(Q)=\int_{\R^d\times\R^d\times\Omega}c(x,y) \,dQ(x,y,\omega)
\]
on $\mathcal M(\R^d\times\R^d\times\Omega)$.
We have the following basic result on existence and uniqueness of
optimal semicouplings, the proof of which is deferred to the Section
\ref{sapp}. The first part of the theorem, the existence and
uniqueness of an optimal semicoupling, is very much in the spirit of an
analogous result by Figalli~\cite{figalli2010optimal} on existence and
(if enough mass is transported) uniqueness of an optimal partial
coupling. However, in our case the second marginal is discrete whereas
in~\cite{figalli2010optimal} it is absolutely continuous.

%
%th2.1 #&#
\begin{theorem}\label{euQ+q}
\textup{(i)} For each bounded Borel set $A\subset\R^d$ there exists a unique
semicoupling $\Q_A$ of $\leb$ and $(1_A\mu^\bullet)\mathbb P$ which
minimizes the mean cost functional
$\Cost(\cdot)$.\vspace*{-6pt}
\begin{longlist}[(iii)]
\item[(ii)] $\Q_A$ can be disintegrated as $d\Q_A(x,y,\omega):=dq_A^\omega
(x,y) \,d{\mathbb P}(\omega)$ where for $\mathbb P$-a.e. $\omega$ the
measure $q_A^\omega$ is the unique minimizer of the cost functional
$\CCost(\cdot)$ among the semicouplings of $\leb$ and $1_A\mu^\omega$.

\item[(iii)] $\Cost(\Q_A)=\int_\Omega\CCost(q_A^\omega) \,d\mathbb
P(\omega).$
\end{longlist}
\end{theorem}

For a bounded Borel set $A\subset\mathbb R^d$, the \emph
{transportation cost on $A$} is given by the random variable $\CCo_{A}\dvtx\Omega\to[0,\infty]$ as
\[
\CCo_{A}(\omega):=\CCost \bigl(q_{A}^\omega
\bigr)=\inf \bigl\{\CCost \bigl(q^\omega \bigr)\dvtx q^\omega
\mbox{ semicoupling of $\leb$ and $1_A \mu^\omega$} \bigr
\}.
\]

%
%le2.2 #&#
\begin{lemma}\label{super}
(1) If $A_1,\ldots,A_n$ are disjoint, then $\forall\omega\in
\Omega$
\[
\CCo_{\bigcup_{i=1}^nA_i}(\omega) \geq\sum_{i=1}^n
\CCo_{A_i}(\omega).\vspace*{-6pt}
\]
\begin{longlist}[(2)]
\item[(2)] If $A_1$ and $A_2$ are translates of each other, then $\CCo_{A_1}$ and $\CCo_{A_2}$ are identically distributed.
\item[(3)] If $A_1,\ldots,A_n$ are disjoint and $\mu^\bullet
(A_1),\ldots
,\mu^\bullet(A_n)$ are independent, then the random variables $\CCo_{A_i},i=1,\ldots,n,$ are independent.
\end{longlist}
\end{lemma}

\begin{pf}
Properties (ii) and (iii) follow directly from the respective
properties of the point process and the invariance of the Lebesgue
measure under translations. The intuitive argument for (i) is that
minimizing the costs on $\bigcup_i A_i$ is more restrictive than doing
it separately on each of the $A_i$. The more detailed argument is the following.
Given any semicoupling $q^\omega$ of $\leb$ and $1_{\bigcup_iA_i}
\mu^\omega$, then for each $i$ the measure
$q_i^\omega:=1_{\R^d\times A_i}q^\omega$ is a semicoupling of $\leb$
and $1_{A_i} \mu^\omega$. Choosing $q^\omega$ as the minimizer of
$\CCo_{\bigcup_{i=1}^nA_i}(\omega)$ yields
\[
\CCo_{\bigcup_{i}A_i}(\omega)=\CCost \bigl(q^\omega \bigr)=\sum
_i\CCost \bigl(q_i^\omega \bigr)\ge
\sum_{i} \CCo_{A_i}(\omega).
\]
\upqed\end{pf}

%s2.7 #&#
\subsection{Convergence along standard exhaustions}\label{standardexhaustion}

%
%f2 #&#
\begin{figure}

\includegraphics{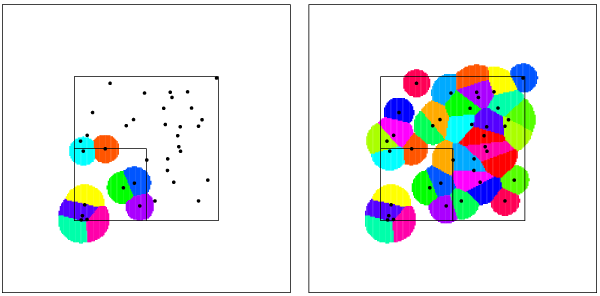}

\caption{Concept of exhausting sequences: start with a small cube and
repeatedly double its edge lengths to exhaust space [cost function
$c(x,y)=|x-y|^2$].}\label{fig2}
\end{figure}

For $n\in\mathbb N_0:=\N\cup\{0\}$ and $z\in\mathbb Z^d$ define the
\emph{cube} or \emph{box} $B_n(z)$ of generation $n$ with basepoint
$z$ by
\[
B_n(z) = z+\bigl[0,2^n\bigr)^d.
\]
For $z=0$ simply put $B_n=B_n(0)$. More generally, for $\gamma=(\gamma_k)\in\Gamma:=(\{0,1\}^d)^{\mathbb N}$ put
\[
B_n(z,\gamma) = z-\sum_{k=1}^n2^{k-1}
\gamma_k+\bigl[0,2^n\bigr)^d.
\]
Starting with the unit box $B_0(z,\gamma)=z+[0,1)^d$, for any random
vector $\gamma\in\Gamma$ the sequence $(B_n(z,\gamma))_{n\in
\mathbb
N_0}$ can be constructed iteratively as follows: Given the box
$B_n(z,\gamma)$ attach $2^d-1$ copies of it---depending on the random
variable $\gamma_{n+1}=(\gamma_{n+1}^1,\ldots,\gamma_{n+1}^d)$ with
values in $\{0,1\}^d$---either on the right (if $\gamma_{n+1}^1=0$) or
on the left (if $\gamma_{n+1}^1=1$), either on the backside (if
$\gamma_{n+1}^2=0$) of on the front (if $\gamma_{n+1}^2=1$), either
on the top
(if $\gamma_{n+1}^3=0$) or on the bottom (if $\gamma_{n+1}^3=1$), etc; see Figure~\ref{fig2}.

The sequence $(B_n(z,\gamma))_{n\in\mathbb N_0}$ for fixed $z$ and
$\gamma$ is increasing and for $\nu$-almost every $\gamma\in\Gamma
$ it
increases to $\mathbb R^d$. Each of the boxes $B_n(z,\gamma)$ contains
the point $z$.

Put
\[
{\mathfrak c}_n:= 2^{-dn}\cdot\EE[\CCo_{B_n(z,\gamma)} ].
\]
Note that translation invariance (equivariance plus stationarity)
implies that the right-hand side does not depend on $z\in\Z^d$ and
$\gamma\in\Gamma$.

%
%co2.3 #&#
\begin{corollary}\label{2cor2}
\textup{(i)} The sequence $({\mathfrak c}_n)_{n\in\mathbb N_0}$ is
nondecreasing. The limit
\[
{\mathfrak c}_\infty= \lim_{n\to\infty}{\mathfrak c}_n =
\sup_n{\mathfrak c}_n
\]
exists in $(0,\infty]$.\vadjust{\goodbreak}
\begin{longlist}[(iii)]
\item[(ii)] Assume that $\mu^\bullet$ is ergodic. Then, we have for all
$z\in\mathbb Z^d$, for all $\gamma\in\Gamma$ and for $\mathbb P$-almost
every $\omega\in\Omega$,
\[
\liminf_{n\to\infty} 2^{-nd}\CCo_{B_n(z,\gamma)}(\omega) = {
\mathfrak c}_\infty.
\]
\item[(iii)] $\mathfrak c_\infty\leq\inf_{q\in\Pi_s}\mathfrak
C_\infty
(q)$ where $\Pi_s$ denotes the set of semicouplings of $\leb$ and
$\mu^\bullet$.
\end{longlist}
\end{corollary}

\begin{pf}
(i) is an immediate consequence of the previous lemma. For (ii) fix an
arbitrary nested sequence of boxes $(B_n)_n$ generated by a standard
exhaustion. Then we have by superadditivity $\forall\omega\in\Omega$
for all $n,k\in\N$
\[
2^{-d(n+k)}\CCo_{B_{n+k}}(\omega)\geq2^{-dk}\sum
_{j=1}^{2^{dk}}2^{-nd}\CCo_{B_n^j}(
\omega),
\]
where $B_n^j$ are disjoint copies of $B_n$ such that $\bigcup_{j=1}^{2^{dk}} B_n^j=B_{n+k}$.
In the limit of $k\to\infty$
we get by ergodicity for $\P$-a.e. $\omega$
\[
\liminf_{k\to\infty}2^{-kd}\CCo_{B_k}(\omega)\geq\EE
\bigl[2^{-nd}\CCo_{B_n} \bigr]={\mathfrak c}_n
\]
for each $n\in\N$ and thus
\[
\liminf_{k\to\infty}2^{-kd}\CCo_{B_k}(\omega)\geq{
\mathfrak c}_\infty.
\]
On the other hand, Fatou's lemma implies
\[
\EE \Bigl[\liminf_{n\to\infty}2^{-nd}\CCo_{B_n} \Bigr]\leq
\liminf_{n\to
\infty}\EE \bigl[2^{-nd}\CCo_{B_n} \bigr]={
\mathfrak c}_\infty.
\]
Both inequalities together imply the assertion.

For (iii) take any semicoupling $q^\bullet$ of $\leb$ and $\mu^\bullet\P
$. Then we have for any n
\[
2^{-dn}\Cost \bigl(1_{R^d\times B_n\times\Omega}q^\bullet \bigr) \geq
\mathfrak c_n.
\]
Taking the limit yields
\[
\mathfrak C_\infty \bigl(q^\bullet \bigr)=
\liminf_{n\to\infty} 2^{-dn} \Cost \bigl(1_{R^d\times B_n\times\Omega}q^\bullet
\bigr) \geq \lim_n \mathfrak c_n=\mathfrak
c_\infty.
\]
\upqed\end{pf}

%
%co2.4 #&#
\begin{corollary}
$\mathfrak c_\infty$ only depends on the scale $\vartheta$ and on the
\emph{distribution} of~$\mu^\bullet$, not on the choice of the
realization of $\mu^\omega$ on a particular probability space
$(\Omega
,\frak A,\mathbb P)$.
\end{corollary}
\begin{pf}
It is sufficient to show that $\mathfrak c_n$ just depends on the
distribution of $\mu^\bullet$. For a given set of points $\varXi
(\omega
)$ in $B_n$ there is a unique semicoupling $q^\omega_{B_n}$ of $\leb$
and $1_{B_n}\mu^\omega$ minimizing $\CCost$; see Proposition \ref
{uniqueq}. Hence, $q^\omega_{B_n}$ just depends on $\varXi(\omega)$.
However, the distribution of the points in $B_n$, $\varXi(\omega)$,
just depends on the distribution of $\mu^\bullet$.
\end{pf}

%
%f3 #&#
\begin{figure}

\includegraphics{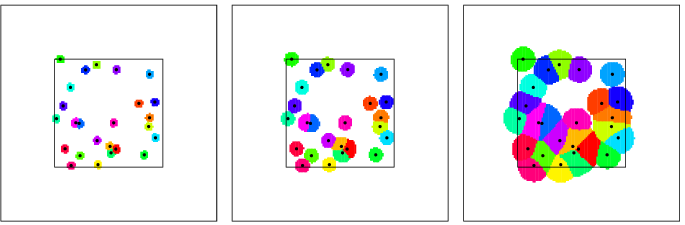}

\caption{Semicoupling of Lebesgue and 25 points in the cube with
$c(x,y)=|x-y|$ where each point gets mass $1/9, 1/3, 1$, respectively.}\label{fig3}
\end{figure}

%
%re2.5 #&#
\begin{remark}
None of the previous definitions and results required that $\mu^\bullet
$ have subunit intensity. However, one easily verifies that
\[
\beta>1 \quad\Longrightarrow\quad{\mathfrak c}_\infty=\infty,
\]
where $\beta:=\EE[\mu^\bullet([0,1)^d) ]$ denotes the
intensity of the equivariant point process.
\end{remark}

%
%re2.6 #&#
\begin{remark}
The problem of finding an optimal semicoupling between $\leb$ and a
Poisson point process $\mu^\bullet$ of intensity $\beta<1$ is
equivalent to the problem
of finding an optimal semicoupling between $\leb$ and $\beta\cdot
\hat\mu^\bullet$ where $\hat\mu^\bullet$ is a Poisson point
process of unit intensity; see Figure~\ref{fig3}.

Indeed, given $\beta\in(0,1)$ and a semicoupling $q^\bullet$ of
$\leb$
and a
Poisson point process $\mu^\bullet$ of intensity $\beta$. Put $\tau
\dvtx
x\mapsto\beta^{1/d} x$ on $\R^d$ as well as on $\R^d\times\R^d$. Then
$\hat\mu^\omega:=\tau_* \mu^\omega$ is a Poisson point process with
intensity 1, and
\[
\tilde q^\omega:=\beta\cdot\tau_*q^\omega
\]
is a semicoupling of $\leb$ and $\beta\cdot\hat\mu^\omega$.
Conversely, given any Poisson point process $\hat\mu^\omega$ of unit
intensity and any semicoupling $\tilde q^\omega$ of $\leb$ and $\beta
\cdot\hat\mu^\omega$,
then $q^\omega:=\frac1\beta\cdot(\tau^{-1})_*\tilde q^\omega$ is a
semicoupling of $\leb$ and $\mu^\omega:=(\tau^{-1})_*\hat\mu^\omega$,
the latter being a Poisson point process of intensity $\beta$.
In both cases, $q$ is equivariant if and only if $\tilde q$ is equivariant.

The asymptotic mean transportation cost for $\tilde q^\bullet$ measured
with scale $\vartheta$ will coincide with the asymptotic mean
transportation cost for $q^\bullet$ measured with scale $\vartheta_\beta
(r):=\beta\cdot\vartheta(\beta^{-1/d} r)$,
\begin{eqnarray*}
\EE\int_{\R^d\times[0,1)^d}\vartheta\bigl(|x-y|\bigr) \,d\tilde q^\bullet=
\EE \int_{\R^d\times[0,1)^d}\vartheta_\beta\bigl(|x-y|\bigr) \,d
q^\bullet.
\end{eqnarray*}
\end{remark}

%s3 #&#
\section{Uniqueness}\label{su}
Throughout this section we fix an equivariant point process $\mu^\bullet
\dvtx\Omega\to\mathcal M(\R^d)$ of subunit intensity and with finite
asymptotic mean transportation cost $\mathfrak c_\infty$.

%
%f4 #&#
\begin{figure}

\includegraphics{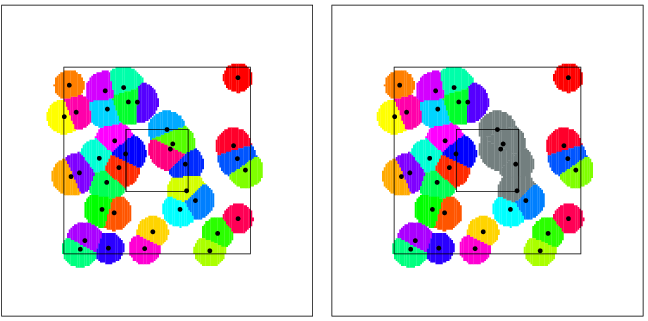}

\caption{The left picture is a semicoupling of Lebesgue and 36 points
with cost function $c(x,y)=|x-y|^4$. In the right picture, the five
points within the small cube can choose new partners from the mass that
was transported to them in the left picture (corresponding to the
measure $\lambda_A$). If the semicoupling on the left-hand side is
locally optimal, then the points in the small
cube on the right-hand side will choose from the gray region exactly
the partners they have in the left picture. }\label{figure4}
\end{figure}

%
%pr3.1 #&#
\begin{proposition}\label{loc-opt}
Given a counting measure $\mu\in\mathcal N(\R^d)$ and a semicoupling
$q$ of $\leb$ and $\mu$, then the following properties are equivalent:
\begin{longlist}[(iii)]
\item[(i)] For each bounded Borel set $A\subset\R^d$, the measure
$1_{\R
^d\times A}q$ is the unique optimal semicoupling of the measures
$\lambda_A(\cdot):=q(\cdot,A)$ and $1_A\mu$; see Figure~\ref{figure4}.
\item[(ii)] The support of $q$ is $c$-cyclically monotone, more precisely,
\[
\sum_{i=1}^n c(x_i,y_i)
\le\sum_{i=1}^n c(x_i,y_{i+1})
\]
for any $n\in\N$ and any choice of points $(x_1, y_1), \ldots, (x_n,
y_n)$ in $\mathrm{supp}(q)$
with the convention $y_{n+1} = y_1$; cf. \eqref{cm}.
\item[(iii)] There exists a density $\rho\dvtx\R^d\to[0,1]$ and a
$c$-cyclically monotone map $T^\omega\dvtx \{\rho>0\}\to\R^d$ such that
%
%
%e4 #&#
\begin{equation}
\label{trans-map} q= (\mathrm{id},T )_*(\rho\leb).
\end{equation}
Recall that, by definition, a map $T$ is $c$-cyclically monotone if and
only if the closure of its graph $\{(x,T(x))\dvtx x\in A^\omega\}$ is a
$c$-cyclically monotone set.
\end{longlist}
\end{proposition}

\begin{pf} The implications $\mathrm{(iii)}\Longrightarrow\mathrm{(ii)}\Longrightarrow
\mathrm{(i)}$ follow from Lemma~\ref{leb-dirac}.

$\mathrm{(i)}\Longrightarrow\mathrm{(iii)}$:
Fix an exhaustion $(B_n')_n$ of $\R^d$ by boxes, say
$B_n'=[-2^{n-1},\break2^{n-1})^d$. For each $n\in\N$, let $\rho_n$ be the
density of the measure $\lambda_n:=\lambda_{B_n'}$ on $\R^d$. This is
the part of Lebesgue measure from which the points inside of $B_n'$
might choose their ``partners.''
Obviously, $0\le\rho_n\le\rho_{n+1}\le1.$ Hence, $\lim_{n\to
\infty} \rho_n(x)=\rho(x)\leq1$ exists $\leb$-a.e.

Assuming (i), according to Proposition~\ref{uniqueq} (or, more
precisely, a canonical extension of it for semicouplings of $\rho\leb$
and $\sigma$),
there exists a
$c$-cyclically monotone map
$T_n\dvtx\{\rho_n>0\}\to\R^d$ such that
\[
dq(x,y) = d\delta_{T_n(x)}(y) \rho_n(x) \,d\leb(x)\qquad \mbox{on }
\R^d\times B_n'.
\]
Since the left-hand side is independent of $n$, we have
\[
T_{n+1} = T_n \qquad\mbox{on } \{\rho_n>0\}.
\]
This trivially yields the existence of
\[
T:=\lim_{n\to\infty}T_n \qquad\mbox{on }\{\rho>0\}:=
\lim_{n\to\infty} \{\rho_n\},
\]
defining a $c$-cyclically monotone map $T\dvtx\{\rho>0\}\to\R^d$
with the
property that
\[
dq(x,y) = d\delta_{T(x)}(y) \rho(x) \,d\leb(x).
\]
\upqed\end{pf}

%
%re3.2 #&#
\begin{remark}\label{trans-map2}
Set $A=\{\rho>0\}.$ In the sequel, any \emph{transport map} $T\dvtx
A\to\R^d$ as above
will be extended to a map $T\dvtx\R^d\to\R^d\cup\{\eth\}$ by putting
$T(x):=\eth$ for all $x\in\R^d\setminus A$ where $\eth$ denotes an
isolated point added to $\R^d$ (``point at infinity,'' ``cemetery'').
Then (\ref{trans-map}) simplifies to
%
%
%e5 #&#
\begin{equation}
q= (\mathrm{id},T )_*(\rho\leb) \qquad\mbox{on }\R^d\times\R^d.
\end{equation}
Moreover, we put $c(x,T(x)) = c(x,\eth):= 0$ for $x\in\R^d\setminus A$.
\end{remark}

%
%de3.3 #&#
\begin{definition}
\begin{itemize}
\item%[$\rhd$]
A semicoupling $Q=q^\bullet\mathbb P$ of $\leb$ and $\mu^\bullet$ is
called \emph{locally optimal} if and only if some (hence every) of the
properties of the previous proposition are satisfied for $\mathbb
P$-a.e. $\omega\in\Omega$.
\item%[$\rhd$]
A semicoupling $Q=q^\bullet\mathbb P$ of $\leb$ and $\mu^\bullet$ is
called \emph{asymptotically optimal} if and only if
\[
\liminf_{n\to\infty}2^{-nd}\Cost(1_{\R^d\times B_n'}Q)=\mathfrak
c_\infty
\]
for some exhaustion $(B_n')_n$ of $\R^d$ by boxes $B_n'=B_n(z,\gamma)$.
\item%[$\rhd$]
A semicoupling $Q=q^\bullet\mathbb P$ of $\leb$ and $\mu^\bullet$
is called
\emph{equivariant} if and only if for each $z\in\Z^d$ the measure $Q$
is equivariant under the diagonal action of $\Z^d$, that is,
\[
q^\omega(A,B) = q^{\omega+z}(A+z,B+z)
\]
for all $z\in\Z^d$ and $A,B\in\mathcal B(\R^d).$
\item%[$\rhd$]
A semicoupling $Q=q^\bullet\mathbb P$ of $\leb$ and $\mu^\bullet$ is
called \emph{optimal} if and only if it is equivariant and
asymptotically optimal.
\end{itemize}
\end{definition}

The very same definitions apply to \emph{couplings} instead of semicouplings.\vadjust{\goodbreak}

%
%re3.4 #&#
\begin{remark}
(i) Asymptotic optimality is not sufficient for uniqueness and it does
not imply local optimality: Given any asymptotically optimal
semicoupling $q^\bullet$ and a bounded Borel set $A\subset\R^d$ of
positive volume, choose an arbitrary coupling $\tilde q_A^\omega$ of
the measures $q^\omega(\cdot,A)$ and $1_A\mu^\omega$, which are the
marginals of $q_A^\omega:=1_{\R^d\times A}q^\omega$. If $\mu^\omega
(A)\ge2$ (which happens with positive probability), then one can always
achieve that $\tilde q_A^\omega$ is a nonoptimal coupling and that it
is different from $q_A^\omega$. Put
\[
\tilde q^\omega:= q^\omega+ \tilde q_A^\omega-
q_A^\omega.
\]
Then $\tilde q^\bullet$ is an asymptotically optimal semicoupling of
$\leb$ and $\mu^\bullet$. It is not locally optimal and it does not
coincide with $q^\bullet$.\vspace*{-6pt}
\begin{longlist}[(iii)]
\item[(ii)] Local optimality does not imply asymptotic optimality and it is
not sufficient for uniqueness: For instance in the case $p=2$, given
any coupling $q^\bullet$ of $\leb$ and $\mu^\bullet$ and $z\in\R^d\setminus\{0\}$, then
\[
d\tilde q^\omega(x,y):=dq^\omega(x+z,y)
\]
defines another locally optimal coupling of $\leb$ and $\mu^\bullet$.
At most one of them can be asymptotically optimal.

\item[(iii)]
Note that local optimality---in contrast to asymptotic optimality and
equivariance---is not preserved under convex combinations. We do not
claim that local optimality and asymptotic optimality imply uniqueness.

\item[(iv)]
Local optimality links classical optimal transportation problems,
problems between finite measures, with optimal transportation problems
between $\leb$ and a point process by locally optimizing the semicouplings.
\end{longlist}
\end{remark}

Given $\gamma,\eta\in\mathcal M(\mathbb R^d)$ with $\gamma(\R^d)\ge\eta
(\R^d)$, we define the \emph{transportation cost} by
\[
\CCost(\gamma,\eta):=\inf \bigl\{\CCost(q)\dvtx q\in\Pi_s(\gamma,
\eta) \bigr\}.
\]
Similarly, given measure valued random variables
$\gamma^\bullet, \eta^\bullet\dvtx \Omega\to\mathcal M(\mathbb R^d)$
and a bounded Borel set $A\subset\R^d$
we define the \emph{mean transportation cost} by
\[
\Cost \bigl(\gamma^\bullet,\eta^\bullet \bigr):=\inf \bigl\{
\Cost \bigl(q^\bullet\mathbb P \bigr)\dvtx q^\omega\in
\Pi_s \bigl(\gamma^\omega,\eta^\omega \bigr)
\mbox{ for a.e. }\omega \bigr\}.
\]

Given a (semi-) coupling $Q=q^\bullet\mathbb P$ of $\leb$ and $\mu^\bullet\P$, recall the definition of $\lambda_A^\bullet$ from
Proposition~\ref{loc-opt}. We define
the \emph{efficiency of the \textup{(}semi-\textup{)} coupling $Q$ on the set~$A$} by
\[
\mathfrak{eff}_{A}(Q):=\frac{\Cost(\lambda^\bullet_A, 1_A\mu^\bullet
)}{\Cost(1_{\R^d\times A}Q)}.
\]
It is a number in $(0,1]$. The (semi-) coupling $Q$ is said to be
efficient on $A$ if and only if $\mathfrak{eff}_{A}(Q)=1$. Otherwise,
it is inefficient on $A$.

%
%le3.5 #&#
\begin{lemma} \textup{(i)} $Q$ is locally optimal if and only if $\mathfrak
{eff}_{A}(Q)=1$ for all bounded Borel sets $A\subset\R^d$.\vspace*{-3pt}
\begin{longlist}[(ii)]
\item[(ii)] $\mathfrak{eff}_{A}(Q)=1$ for some $A\subset\R^d$ implies
$\mathfrak{eff}_{A'}(Q)=1$ for all $A'\subset A$.\vadjust{\goodbreak}
\end{longlist}
\end{lemma}

\begin{pf} (i) Let $A$ be given and $\omega\in\Omega$ be fixed. Then
$1_{\R^d\times A}q^\omega$ is the optimal semicoupling of the measures
$\lambda^\omega_A$ and $1_A\mu^\omega$ if and only if
%
%
%e6 #&#
\begin{equation}
\label{eff-path} \CCost \bigl(1_{\R^d\times A}q^\omega \bigr)=\CCost
\bigl(\lambda_A^\omega, 1_A
\mu^\omega \bigr).
\end{equation}
On the other hand, $\mathfrak{eff}_{A}(Q)=1$ is equivalent to
\[
\EE \bigl[\CCost \bigl(1_{\R^d\times A}q^\bullet \bigr) \bigr]=\EE
\bigl[ \CCost \bigl(\lambda_A^\bullet, 1_A
\mu^\omega \bigr) \bigr].
\]
The latter, in turn, is equivalent to (\ref{eff-path}) for $\mathbb
P$-a.e. $\omega\in\Omega$.%\vspace*{-6pt}
\begin{longlist}[(ii)]
\item[(ii)]If the transport $q$ restricted to $\R^d\times A$ is optimal,
then also each of its sub-transports; see Theorem 4.6 in \cite
{villani2009optimal}.\quad\qed
\end{longlist}
\noqed\end{pf}

%
%th3.6 #&#
\begin{theorem} Every optimal semicoupling of $\leb$ and $\mu^\bullet
\P
$ is locally optimal.
\end{theorem}

\begin{pf}
Assume we are given a semicoupling $Q$ of $\leb$ and $\mu^\bullet\P$
which is equivariant and not locally optimal. According to the previous
lemma, the latter implies that there exist $n\in\N$ and $z_0\in\Z^d$
such that the semicoupling $Q$ is not efficient on the box
$B_n(z_0)=z_0+[0,2^n)^d$, that is,
\[
\eta:=\mathfrak{eff}_{B_n(z_0)}(Q)<1.
\]
By equivariance this implies
$\mathfrak{eff}_{B_n(z)}(Q)=\eta<1$ for all $z\in\Z^d$.
Hence, for each $z\in\Z^d$ there exists a measure-valued random
variable $\tilde q^\bullet_{B_n(z)}$ such that $\tilde q^\omega_{B_n(z)}$ for a.e. $\omega$ is a semicoupling of
$\lambda^\omega_{B_n(z)}$ and $1_{B_n(z)}\mu^\omega$ and more
efficient than
$q^\omega_{B_n(z)}:=1_{\R^d\times{B_n(z)}}\cdot q^\omega$, that is,
such that
\[
\EE \bigl[\CCost \bigl(\tilde q^\bullet_{B_n(z)} \bigr) \bigr]
\le \eta\cdot\EE \bigl[\CCost \bigl( q^\bullet_{B_n(z)} \bigr)
\bigr].
\]
Put
\[
\tilde q^\bullet:=\sum_{z\in(2^n \Z)^d}\tilde
q^\bullet_{B_n(z)}.
\]
Then $\tilde q^\bullet$ is a semicoupling of $\leb$ and $\mu^\bullet$
and for all $z\in(2^n \Z)^d$
\[
\EE \bigl[\CCost \bigl(1_{\R^d\times{B_n(z)}}\tilde q^\bullet \bigr) \bigr]
\le \eta\cdot\EE \bigl[\CCost \bigl(1_{\R^d\times{B_n(z)}} q^\bullet \bigr)
\bigr].
\]
Equivariance of $q^\bullet$---together with uniqueness of cost
minimizers on bounded sets---implies equivariance of $\tilde q^\bullet$
under the group $(2^n\Z^d)$. In other words, $\tilde Q=\tilde
q^\bullet
\mathbb P$ is an $(2^n\Z^d)$-equivariant semicoupling of $\leb$ and
$\mu^\bullet\mathbb P$ which satisfies
\[
\Cost(1_{\R^d\times B_n(z)}\tilde Q) \le\eta\cdot\Cost(1_{\R
^d\times B_n(z)}Q)
\]
for all $z\in(2^n \Z)^d$. Additivity of the mean cost functional
$\Cost
(\cdot)$ implies
\[
\Cost(1_{\R^d\times B_{n+k}} \tilde Q) \le\eta\cdot\Cost(1_{\R
^d\times B_{n+k}} Q)
\]
for all $k\in\N_0$ and therefore, due to Corollary~\ref{2cor2}(iii), finally
\[
\mathfrak c_\infty\le\liminf_{k\to\infty}\Cost(1_{\R^d\times
B_{k}}
\tilde Q) \le\eta\cdot\liminf_{k\to\infty} \Cost(1_{\R
^d\times B_{k}} Q)
\]
with $\eta<1$.
This proves that $Q$ is not asymptotically optimal.\vadjust{\goodbreak}
\end{pf}

%
%le3.7 #&#
\begin{lemma}
Let $q^\omega=(\mathrm{id},T^\omega)_*(\rho^\omega\leb)$ be an optimal
semicoupling between $\leb$ and $\mu^\bullet.$ Then, $\P$-a.s. we have
$\rho^\omega(x)\in\{0,1\} \leb$-a.e.
\end{lemma}
\begin{pf}
Assume there is a $n\in\N$ and $B_n(z_0)=z_0+[0,2^n)^d$ such that on a
set of positive $\P$-measure
\[
q^\omega_n:=1_{\R^d\times B_n(z_0)} \,d q^\omega(x,y)=
\bigl(\mathrm{id},T^\omega \bigr)_* \bigl(\rho_n^\omega
\leb \bigr)
\]
with $0<\rho^\omega_n<1$ on a set of positive $\leb$-measure. However,
due to Proposition~\ref{uniqueq} this implies that $Q=q^\bullet\P$ is
not efficient on $B_n(z_0)$ because it is possible to construct a
semicoupling between $1_{\rho_n^\omega>0}\leb$ and $1_{B_n(z_0)}\mu^\omega$ with less cost. By the same reasoning as in the last proof,
this implies that $Q$ is not optimal.
\end{pf}

Hence, any optimal semicoupling can be written as $q^\omega
=(\mathrm{id},T^\omega
)_*\leb$ for some measurable map $T\dvtx A^\omega\to\R^d\cup\{\eth
\}$; cf
Remark~\ref{trans-map2}.

%
%th3.8 #&#
\begin{theorem}\label{uniqueness}
There exists at most one optimal semicoupling of $\leb$ and~$\mu^\bullet\P$.
\end{theorem}

\begin{pf} Assume we are given two optimal semicouplings $q_1^\bullet$
and $q_2^\bullet$. Then also $q^\bullet:=\frac12 q_1^\bullet+\frac12
q_2^\bullet$ is an optimal semicoupling. Hence, by the previous theorem
all three couplings---$q_1^\bullet$, $q_2^\bullet$ and $q^\bullet
$---are locally optimal. Thus, for a.e. $\omega$ by the results of
Proposition~\ref{loc-opt} and the last lemma there exist maps
$T_1^\omega,T_2^\omega,T^\omega$ and sets $A_1^\omega, A_2^\omega,
A^\omega$ such that
\begin{eqnarray*}
dq^\omega(x,y)&=&d\delta_{T^\omega(x)}(y) 1_{A^\omega}(x) \,d
\leb(x)
\\
&=& \bigl(\tfrac12d\delta_{T_1^\omega(x)}(y)1_{A^\omega
_1}(x)+\tfrac12d
\delta_{T_2^\omega(x)}(y)1_{A^\omega_2}(x) \bigr) \,d\leb(x).
\end{eqnarray*}
This, however, implies $T_1^\omega(x)=T_2^\omega(x)$ for a.e. $x\in
A_1^\omega\cap A_2^\omega$ and, moreover, $A_1^\omega=A_2^\omega$. Thus
$q_1^\omega=q_2^\omega$.
\end{pf}

%
%re3.9 #&#
\begin{remark}
Note that we only used equivariance under the action of~$\Z^d$.
However, the minimizer is equivariant under the action of $\R^d$. For
the uniqueness it would also have been sufficient to require
equivariance under the action of $k\Z^d$ for some $k\in\N$.
\end{remark}

%
%th3.10 #&#
\begin{theorem}\label{coupling}\label{semi=opt}
\textup{(i)} If $\mu^\bullet$ has unit intensity, then every optimal
semicoupling of $\leb$ and $\mu^\bullet$ is indeed a coupling of them.\vspace*{-3pt}
\begin{longlist}[(ii)]
\item[(ii)] Conversely, if an optimal coupling exists, then $\mu^\bullet$ must
have unit intensity.
\end{longlist}
\end{theorem}
This theorem is in a similar spirit as Theorem 4 in~\cite{stable-marriage}.
\begin{pf}
(i) Let $Q$ be an optimal semicoupling.
For $n\in\N$ put $B_n(z)=z+[0,2^n)^d$ and consider the \emph
{saturation} $\alpha_k:=2^{-kd}Q(B_k(z)\times B_k(z)\times\Omega
)\leq
1$. Note that\vadjust{\goodbreak} $\alpha_k$ is independent of $z\in\Z^d$. Hence, we have
$\alpha_k\leq\alpha_{k+1}$. Indeed, $B_{k+1}(z)$ is the disjoint union
of $2^d$ cubes $B_k(y_j)$ for suitable $y_j$. Therefore,
\[
\alpha_{k+1}\geq2^{-d}\sum_{j=1}^{2^d}2^{-kd}
Q \bigl(B_k(y_j)\times B_k(y_j)
\times\Omega \bigr)=\alpha_k.
\]
Thus, the limit $\alpha_\infty:=\lim_{k\to\infty}\alpha_k$
exists, and
we have $\alpha_\infty\in(0,1]$.

Since $\mu^\bullet$ has unit intensity and since $Q$ is a semicoupling,
we have $Q(\R^d\times B_k\times\Omega)=2^{kd}$. Let us first assume
that $\alpha_\infty<1$ and choose $r=[(1+\frac12(1-\alpha_\infty
))^{1/d}-1]/2$. Then for all $k\in\mathbb N$ mass of a total amount of
at least $(1-\alpha_\infty)2^{kd}$ has to be transported from
$\complement B_k$ into $B_k$.
The volume of the $(r2^k)$-neighborhood of the box $B_k$ is less than
$\frac12(1-\alpha_\infty)2^{kd}$. Hence, mass of total amount of at
least $\frac{1}{2}(1-\alpha_\infty)2^{kd}$ has to be transported at
least the distance~$r2^k$. Thus, we can estimate the costs per unit
from below by
\begin{eqnarray*}
2^{-kd}\int_{\R^d\times B_k\times\Omega}c(x,y) \,dQ(x,y,\omega)\geq
\frac{1}{2}(1-\alpha_\infty)\vartheta \bigl(r2^k
\bigr).
\end{eqnarray*}
The right-hand side diverges as $k$ tends to infinity which contradicts
the finiteness of the costs per unit. Thus, we have $\alpha_\infty=1$.
Furthermore, for all $k$ there is a $u\in B_k(0)$ such that
\begin{eqnarray*}
\alpha_k&=&2^{-kd}Q \bigl(B_k(0)\times
B_k(0)\times\Omega \bigr)
\\
&=&2^{-kd}\sum_{v\in B_k(0)\cap\Z^d}Q
\bigl(B_0(v) \times B_k(0)\times\Omega \bigr)
\\
&\leq& Q \bigl(B_0(u)\times B_k(0)\times\Omega \bigr)
\le Q \bigl(B_0(u)\times\R^d\times\Omega \bigr).
\end{eqnarray*}
However, by translation invariance (equivariance plus stationarity) the
quantity $Q(B_0(u)\times\R^d\times\Omega)$ is independent of $u$.
Moreover, it is bounded above by 1 as $Q$ is a semicoupling. Hence, we
have for all $v\in\R^d$:
\[
1=\limsup_{k\to\infty}\alpha_k\le Q \bigl(B_0(v)
\times\R^d\times\Omega \bigr)\le1.
\]
Therefore, $Q$ is actually a coupling of the Lebesgue measure and the
point process.

(ii) Assume that $Q$ is an optimal coupling and that $\beta<1$. Then a
similar argumentation as above yields that for each box $B_k$, Lebesgue
measure of total mass $\ge(1-\beta)\cdot2^{kd}$ has to be transported
from the interior of $B_k$ to the exterior. As $k$ tends to $\infty$,
the costs of these transports explode.
\end{pf}

%
%co3.11 #&#
\begin{corollary} In the case $\vartheta(r)=r^2$, given an optimal
coupling $q^\bullet$ of $\leb$ and a point process $\mu^\bullet$ of
unit intensity then for a.e. $\omega\in\Omega$ there exists a convex
function $\varphi^\omega\dvtx\R^d\to\R$ (unique up to additive constants)
such that
\[
q^\omega= \bigl(\mathrm{id},\nabla\varphi^\omega \bigr)_*\leb.
\]
In particular, a ``fair allocation rule'' is given by the \emph{monotone
map} \mbox{$T^\omega=\nabla\varphi^\omega$}.\vadjust{\goodbreak}

Moreover, for a.e. $\omega$ and any center $\xi\in\Xi(\omega
):=\mathrm
{supp}(\mu^\omega)$, the associated cell
\[
S^\omega(\xi) = { \bigl(T^\omega \bigr)^{-1} \bigl(
\{ \xi\} \bigr)}
\]
is a convex polyhedron of volume $\mu^\omega(\xi)\in\N$. If the point
process is simple, then all these cells have volume 1.
\end{corollary}

\begin{pf}
By Proposition~\ref{loc-opt} we know that $T^\omega=\lim_{n\to
\infty
}T_n^\omega$, where $T_n^\omega$ is an optimal transportation map from
some set $A_n^\omega$ to $B_n^{\prime}$. From the classical theory (see~\cite
{Brenier,GangboMcCann1996}), we know that $T_n^\omega=\nabla\varphi_n^\omega$ for some convex function $\varphi_n^\omega$. More precisely,
\[
\varphi_n^\omega(x)=\max_{\xi\in\varXi(\omega)\cap B_n^{\prime}}
\bigl(x^2-\abs{x-\xi}^2/2+b_\xi \bigr)
\]
for some constants $b_\xi$. Moreover, we know that $T_{n+k}^\omega
=T_n^\omega$ on $A_n^\omega$ for any $k\in\N$. Fix any $\xi_0\in
\Xi
(\omega)$. Then there is $n\in\N$ such that $\xi_0\in B_n^{\prime}$. Then
$(T_{n+k}^\omega)^{-1}(\xi_0)=(T_{n}^\omega)^{-1}(\xi_0)$ for any
$k\in
\N$. Furthermore,
\begin{eqnarray}
T_n^\omega(x)=\xi_0
\quad
\Leftrightarrow\quad
-\abs{x-\xi_0}^2/2+b_{\xi_0} > -\abs{x-
\xi}^2/2+b_\xi\nonumber\\
\eqntext{\forall\xi\in\Xi(\omega)\cap
B_n^{\prime}, \xi\neq\xi_0.}
\end{eqnarray}
For fixed $\xi\neq\xi_0$ this equation describes two half-spaces
separated by a hyperplane (defined by equality in the equation above).
The set $S^\omega(\xi_0)$ is then given as the intersection of all
these halfspaces defined by $\xi_0$ and $\xi\in\Xi(\omega)\cap
B_n^{\prime}$. Hence, it is a convex polytope. Moreover, the last inequality
is exactly the defining equation for a Laguerre tessellation wrt $\supp
(\mu^\omega)$ and weights $b_\xi$; see~\cite{lautensack2007}.
\end{pf}

%s4 #&#
\section{Construction of optimal semicouplings}\label{sc}
Again we fix an equivariant point process $\mu^\bullet\dvtx\Omega
\to
\mathcal M(\R^d)$ of subunit intensity and with finite asymptotic mean
transportation cost $\mathfrak c_\infty$.

%s4.1 #&#
\subsection{Second randomization and annealed limits}
The crucial step in our construction of an optimal coupling of Lebesgue
measure and the point process will be the introduction of a \emph
{second randomization}, in addition to the first randomness modeled on
the probability space $(\Omega, \mathfrak{A}, \mathbb P)$ which
describes the random choice $\omega\mapsto\mu^\omega$ of a realization
of the point process. The second randomization describes the random
choice $\gamma\mapsto(B_n(z,\gamma) )_{n\in\N}$ of an
increasing sequence of boxes containing a given starting point $z\in\Z^d$; see also Section~\ref{standardexhaustion}. It is modeled on the
\emph{Bernoulli scheme} $(\Gamma,\mathfrak B(\Gamma),\nu)$ with
$\Gamma
=(\{0,1\}^d)^{\mathbb N}$, $\mathfrak B(\Gamma)$ its Borel $\sigma
$-field and $\nu$ the {uniform distribution} on $\Gamma=(\{0,1\}^d)^{\mathbb N}$ (or, more precisely, the infinite product of the
uniform distribution on $\{0,1\}^d$).

For each $z\in\mathbb Z^d,\gamma\in\Gamma$ and $k\in\mathbb N$, recall
that $Q_{B_k(z,\gamma)}$ denotes the minimizer of $\Cost$ among the
semicouplings of $\leb$ and $(1_{B_k(z,\gamma)} \mu^\bullet)\mathbb P$
as constructed in Theorem~\ref{euQ+q}.
Equivariance of this minimizer implies that
\[
\Q_{B_k(z',\gamma)}(A,B,\omega) = \Q_{B_k(z,\gamma
)} \bigl(A+z-z',B+z-z',
\omega+z-z' \bigr)
\]
for all $z,z'\in\mathbb Z^d$ and $A,B\in\mathcal B(M)$.
Put
\[
\d\Q_z^k(x,y,\omega):= \int_\Gamma
\, \d\Q_{B_k(z,\gamma)}(x,y,\omega)\,\d\nu(\gamma)
\]
and
$\d\dot\Q_z^k(x,y,\omega):= 1_{B_0(z)}(y)\,\d\Q_z^k(x,y,\omega)$.

The measure $\dot\Q_z^k$ defines a semicoupling between the Lebesgue
measure and the point process restricted to the box $B_0(z)$. It is a
deterministic, fractional allocation in the following sense:
\begin{itemize}
\item it is a deterministic function of $\mu^\omega$ and does not
depend on any additional randomness [coming, e.g., from $d\nu(\gamma$)];
\item the measure transported into a given point of the point process
has density $\le1$.
\end{itemize}
The last fact of course implies that the semicoupling $\dot\Q_z^k$ is
\emph{not} optimal. The first fact implies that all the objects derived
from $\dot\Q_z^k$ in the sequel---like $\dot\Q_z^\infty$ and $\Q^\infty
$---are also deterministic.

%
%le4.1 #&#
\begin{lemma}\label{Qz-tight}
\textup{(i)} For each $k\in\mathbb N$ and $z\in\mathbb Z^d$
\[
\int_{\mathbb R^d\times B_0(z)\times\Omega}c(x,y)\,\d\Q_z^k(x,y,
\omega)\le{\mathfrak c_\infty}.\vspace*{-6pt}
\]
\begin{longlist}[(iii)]
\item[(ii)] The family $(\dot\Q_z^k)_{k\in\mathbb N}$ of probability
measures on $\mathbb R^d\times\mathbb R^d\times\Omega$ is relatively
compact in the weak topology.
\item[(iii)] There exist probability measures $\dot\Q_z^\infty$ and a
subsequence $(k_l)_{l\in\mathbb N}$ such that for all $z\in\mathbb Z^d$
\[
\dot\Q^{k_l}_z \longrightarrow\dot\Q_z^\infty\qquad
\mbox{weakly as $l\to\infty$.}
\]
\end{longlist}
\end{lemma}

\begin{pf}
(i) Let us fix $z\in\mathbb Z^d$ and start with the following
important observation: \emph{For given $n\in\mathbb N$ the initial box
$B_0(z)$ has each possible ``relative position within $B_n(z,\gamma)$''
with equal probability.}

Hence, together with translation invariance of $\Q_{B_k(z,\gamma)}$
(which in turn follows from equivariance and stationarity of $\P$) we obtain
\begin{eqnarray*}
&&\int_{\mathbb R^d\times B_0(z)\times\Omega}c(x,y)\,\d\Q_z^k(x,y,
\omega)
\\
&&\qquad=\int_{\Gamma}\int_{\mathbb R^d\times B_0(z)\times\Omega
}c(x,y)\,\d
\Q_{B_k(z,\gamma)}(x,y,\omega)\,\d\nu(\gamma)
\\
&&\qquad=2^{-kd}\sum_{v\in B_k(z)\cap\mathbb Z^d} \biggl[\int
_{\mathbb
R^d\times B_0(v)\times\Omega}c(x,y)\,\d\Q_{B_k(z)}(x,y,\omega) \biggr]
\\
&&\qquad=2^{-kd}\int_{\mathbb R^d\times B_k(z)\times\Omega}c(x,y)\,\d\Q_{B_k(z)}(x,y,
\omega)
\\
&&\qquad={\mathfrak c}_k \le{\mathfrak c}_\infty.
\end{eqnarray*}

(ii) In order to prove tightness of $(\dot\Q_z^k)_{k\in\mathbb N}$, let
\[
K_m:= \Bigl\{y\in\mathbb R^d\dvtx\inf_{x\in B_0(z)}|x-y|
\le m \Bigr\}
\]
denote the closed $m$-neighborhood of the unit box based at $z$.
Then
\begin{eqnarray*}
Q_z^k \bigl(\complement K_m\times
B_0(z)\times\Omega \bigr)\le \frac1{\vartheta(m)}\int
_{\mathbb R^d\times B_0(z)\times\Omega}c(x,y)\,\d\Q_z^k(x,y,\omega)
\le\frac1{\vartheta(m)}\cdot{\mathfrak c}_\infty.
\end{eqnarray*}
Since $\vartheta(m)\to\infty$ as $m\to\infty$ this proves
tightness of
the family $(\dot\Q_z^k)_{k\in\mathbb N}$
on $\mathbb R^d\times\mathbb R^d\times\Omega$.
(Recall that $\Omega$ was assumed to be compact from the very beginning.)

(iii)
Tightness yields the existence of $\dot\Q_z^\infty$ and of a converging
subsequence for each $z$. A standard argument (``diagonal sequence'')
then gives convergence for all $z\in\mathbb Z^d$ along a common subsequence.
\end{pf}

%
%le4.2 #&#
\begin{lemma}\label{Qz-disjoint}
\textup{(i)} For each $r>0$ there exist numbers $\varepsilon_k(r)$ with
$\varepsilon_k(r)\to0$ as $k\to\infty$ such that
for all $z,z'\in\mathbb Z^d$ and all $k\in\mathbb N$
\begin{eqnarray*}
&&\int_\Gamma Q_{B_k(z',\gamma)}(A) \,d\nu(\gamma)
\\
&&\qquad\le \int_\Gamma Q_{B_k(z,\gamma)}(A) \,d\nu(\gamma) +
\varepsilon_k \bigl(\bigl|z-z'\bigr| \bigr)\cdot
\sup_\gamma Q_{B_k(z',\gamma)}(A)
\end{eqnarray*}
for any Borel set $A\subset\mathbb R^d\times\mathbb R^d\times\Omega$.

\textup{(ii)} For all $z_1,\ldots,z_m\in\mathbb Z^d$, all $k\in\mathbb N$ and
all Borel sets $A\subset\mathbb R^d$,
\[
\sum_{i=1}^m \dot\Q_{z_i}^k
\bigl(A\times\mathbb R^d\times\Omega \bigr)\le \Biggl(1+\sum
_{i=1}^m\varepsilon_k\bigl(|z_1-z_i|\bigr)
\Biggr)\cdot\leb(A).
\]
\end{lemma}

\begin{pf}
(i) First, note that for each $z,z'\in\mathbb Z^d,k\in\mathbb
N,\gamma
\in\Gamma$,
\[
z'\in B_k(z,\gamma) \quad\iff\quad\exists\gamma'
\dvtx B_k(z,\gamma)=B_k \bigl(z',
\gamma' \bigr)
\]
and in this case
\[
\nu \bigl( \bigl\{\gamma' \dvtx B_k
\bigl(z', \gamma' \bigr)=B_k(z,\gamma)
\bigr\} \bigr) = 2^{-kd}.
\]
Moreover,
\[
\nu \bigl( \bigl\{\gamma\dvtx z'\notin B_k(z,\gamma)
\bigr\} \bigr) \leq\varepsilon_k \bigl(\bigl|z-z'\bigr| \bigr)
\]
for some $\varepsilon_k(r)$ with $\varepsilon_k(r)\to0$ as $k\to
\infty
$ for each $r>0$.
It implies that for each pair $z,z'\in\mathbb Z^d$ and each $k\in
\mathbb N$,
\[
\nu \bigl( \bigl\{\gamma\in\Gamma\dvtx\exists\gamma' \dvtx
B_k(z,\gamma)=B_k \bigl(z',
\gamma' \bigr) \bigr\} \bigr) \geq1-\varepsilon_k
\bigl(\bigl|z-z'\bigr| \bigr).
\]
Therefore, for each Borel set $A\subset\mathbb R^d\times\mathbb
R^d\times\Omega$,
\begin{eqnarray*}
&&\int_\Gamma Q_{B_k(z',\gamma)}(A) \,d\nu(\gamma)
\\
&&\qquad \le \int_\Gamma Q_{B_k(z,\gamma)}(A) \,d\nu(\gamma) +
\varepsilon_k \bigl(\bigl|z-z'\bigr| \bigr)\cdot
\sup_\gamma Q_{B_k(z',\gamma)}(A).
\end{eqnarray*}

(ii)
According to the previous part (i), for each Borel set $A\subset
\mathbb R^d$,
\begin{eqnarray*}
&&\sum_{i=1}^m \dot
\Q^k_{z_i} \bigl(A\times\mathbb R^d\times\Omega
\bigr)
\\
&&\qquad= \sum_{i=1}^m \int
_\Gamma Q_{B_k(z_i,\gamma)} \bigl(A\times B_0(z_i)
\times\Omega \bigr) \,d\nu(\gamma)
\\
&&\qquad\le \sum_{i=1}^m \biggl[\int
_\Gamma Q_{B_k(z_1,\gamma)} \bigl(A\times B_0(z_i)
\times\Omega \bigr) \,d\nu(\gamma)
\\
&&\hspace*{16pt}\qquad\quad{} + \varepsilon_k\bigl(|z_1-z_i|\bigr)\cdot
\sup_{\gamma\in\Gamma}Q_{B_k(z_i,\gamma)} \bigl(A\times B_0(z_i)
\times\Omega \bigr) \biggr]
\\
&&\qquad\le Q_{B_k(z_1,\gamma)} \bigl(A\times\mathbb R^d\times\Omega \bigr)+
\sum_{i=1}^m \varepsilon_k\bigl(|z_1-z_i|\bigr)
\cdot\leb(A)
\\
&&\qquad\le \Biggl( 1+\sum_{i=1}^m
\varepsilon_k\bigl(|z_1-z_i|\bigr) \Biggr)\cdot\leb(A).
\end{eqnarray*}
\upqed\end{pf}

%
%th4.3 #&#
\begin{theorem}\label{th4.3}
The measure $Q^\infty:=\sum_{z\in\mathbb Z^d}\dot\Q_z^\infty$ is an
optimal semicoupling of $\leb$ and $\mu^\bullet$.
\end{theorem}

\begin{pf}
(i) \textit{Second/third marginal}: For any $f\in\mathcal
C_b^+ (\R^d\times\Omega)$ we have due to Lemma~\ref{Qz-tight},
\begin{eqnarray*}
&& \int_{\R^d\times\Omega} f(y,\omega)\,d Q^\infty(x,y,\omega)
\\
&&\qquad=\sum_{z\in\Z^d} \int_{\R^d\times\Omega} f(y,
\omega) \,d\dot Q_z^\infty(x,y,\omega)
\\
&&\qquad=\sum_{z\in\Z^d} \lim_{l\to\infty} \int
_{\R^d\times\Omega} f(y,\omega) \,d\dot Q^{k_l}_z(x,y,
\omega)
\\
&&\qquad=\sum_{z\in\Z^d} \int_{\R^d\times\Omega}f(y,
\omega)1_{B_0(z)}(y) \,d \bigl(\mu^\bullet\P \bigr) (y,\omega)
\\
&&\qquad=\int_{\R^d\times\Omega}f(y,\omega) \,d \bigl(\mu^\bullet\P
\bigr) (y,\omega).
\end{eqnarray*}

(ii) \textit{First marginal}: Let an arbitrary bounded open set
$A\subset\mathbb R^d$ be given, and let $(z_i)_{i\in\mathbb N}$ be an
enumeration of $\mathbb Z^d$.
According to the previous Lemma~\ref{Qz-disjoint}, for any $m\in
\mathbb
N$ and any $k\in\mathbb N$,
\[
\sum_{i=1}^m \dot\Q_{z_i}^k
\bigl(A\times\mathbb R^d\times\Omega \bigr)\le \Biggl(1+\sum
_{i=1}^m\varepsilon_k\bigl(|z_1-z_i|\bigr)
\Biggr)\cdot\leb(A).
\]
Letting first $k$ tend to $\infty$ yields
\[
\sum_{i=1}^m \dot\Q_{z_i}^\infty
\bigl(A\times\mathbb R^d\times\Omega \bigr)\le\leb(A).
\]
Then with $m\to\infty$ we obtain
\[
Q^\infty \bigl(A\times\mathbb R^d\times\Omega \bigr)\le
\leb(A),
\]
which proves that $(\pi_1)_*Q^\infty\le\leb$.

(iii) \textit{Optimality}: By construction, $Q^\infty$ is $\Z^d$-equivariant. Due to the stationarity of $\P$, the asymptotic cost
is given by
\begin{eqnarray*}
&&\int_{\mathbb R^d\times B_0(0)\times\Omega}c(x,y) \,dQ^\infty (x,y,\omega)
\\
&&\qquad= \sum_{z\in\mathbb Z^d} \int_{\mathbb R^d\times B_0(0)\times
\Omega}c(x,y)
\,d\dot\Q_z^\infty(x,y,\omega)
\\
&&\qquad= \int_{\mathbb R^d\times B_0(0)\times\Omega}c(x,y) \,d\dot\Q_0^\infty
(x,y,\omega) \le\mathfrak c_\infty.
\end{eqnarray*}
Here the final \emph{inequality} is due to Lemma~\ref{Qz-tight},
property (i) (which remains true in the limit $k=\infty$), and the last
\emph{equality} comes from the fact that
\[
\int_{\mathbb R^d\times B_0(u)\times\Omega}c(x,y) \,d\dot\Q_z^k(x,y,
\omega)=0
\]
for all $z\not= u$ and for all $k\in\mathbb N$ (which also remains true
in the limit $k=\infty$).
\end{pf}

%
%co4.4 #&#
\begin{corollary}
\textup{(i)} For $k\to\infty$, the sequence of measures $\Q^k:=\break   \sum_{z\in
\Z
^d} \dot\Q^k_z$, $k\in\N$, converges vaguely to the unique optimal
semicoupling~$\Q^\infty$.\vadjust{\goodbreak}

\begin{longlist}[(ii)]
\item[(ii)] For each $z\in\Z^d$ the sequence $(Q^k_z)_{k\in\N}$ converges
vaguely to the unique optimal semicoupling~$\Q^\infty$.
\end{longlist}
\end{corollary}
\begin{pf}
(i) A slight extension of the previous Lemma~\ref{Qz-tight}(iii)${}+{}$Theo\-rem~\ref{th4.3}
yields that each subsequence $(Q^{k_n})_n$ of the above sequence
$(Q^k)_k$ will have a sub-subsequence converging vaguely to an optimal
coupling of $\leb$ and $\mu^\bullet$. Since the optimal coupling is
unique, all these limit points coincide. Hence, the whole sequence
$(Q^k)_k$ converges to this limit point; see, for example,~\cite
{dudley2002real}, Proposition~9.3.1.

(ii) Lemma~\ref{Qz-disjoint}(i) implies that for $z,z',u\in\Z^d$ and every
measurable $A\subset\R^d\times\R^d\times\Omega$,
\begin{eqnarray*}
&&\bigl|Q_z^k \bigl(A\cap \bigl(\R^d
\times B_0(u)\times\Omega \bigr) \bigr)-Q_{z'}^k
\bigl(A\cap \bigl(\R^d \times B_0(u)\times\Omega \bigr)
\bigr)\bigr|
\\
&&\qquad\le \varepsilon_k \bigl(\bigl|z-z'\bigr| \bigr)\cdot
\sup_{v\in\Z^d} Q_{B_k(v)} \bigl(A\cap \bigl(\R^d \times
B_0(u)\times\Omega \bigr) \bigr)
\\
&&\qquad\le\varepsilon_k \bigl(\bigl|z-z'\bigr| \bigr)\to0
\end{eqnarray*}
as $k\to\infty$. Hence, for each $f\in\mathcal C_c(\R^d\times\R^d\times
\Omega)$ and each $z'\in\R^d$,
\[
\biggl\llvert\sum_{z\in\Z^d} \int f(x,y,\omega)
1_{B_0(z)}(y) \,dQ_z^k-\int f(x,y,\omega)
\,dQ^k_{z'} \biggr\rrvert\to0.
\]
That is,
$|\int f \,dQ^k-\int f \,dQ^k_{z'}|\to0$ as $k\to\infty$.
\end{pf}

%
%co4.5 #&#
\begin{corollary}
We have $\mathfrak c_\infty= \inf_{q^\bullet\in\Pi_s}\mathfrak
C_\infty(q^\bullet)$ where $\Pi_s$ denotes the set of all semicouplings
$q^\bullet$ of $\leb$ and $\mu^\bullet$. In particular, the
following holds:
\begin{eqnarray*}
&&\inf_{q^\bullet\in\Pi_s}  \liminf_{n\to\infty} \frac1{\leb (B_n)}\EE
\biggl[\int_{\R^d\times B_n}c(x,y) \,dq^\bullet(x,y) \biggr]
\\
&&\qquad= \liminf_{n\to\infty} \inf_{q^\bullet\in\Pi_s} \frac1{\leb (B_n)}
\EE \biggl[\int_{\R^d\times B_n}c(x,y) \,dq^\bullet(x,y) \biggr].
\end{eqnarray*}
\end{corollary}
\begin{pf}
The optimal coupling $\Q$ constructed in the previous theorem has mean
asymptotic transportation cost bounded above by $\mathfrak c_\infty$.
Thus, we have $\inf_{q^\bullet\in\Pi_s}\mathfrak C_\infty
(q^\bullet
)\leq\mathfrak c_\infty$. Together with Lemma~\ref{2cor2}, this yields
the claim.
\end{pf}

%s4.2 #&#
\subsection{Quenched limits}

According to Section~\ref{su}, the unique optimal semicoupling between $d\leb
(x)$ and $d\mu^\omega(y) \,d\P(\omega)$ can be represented on $\R^d\times
\R^d\times\Omega$ as
\[
dQ^\infty(x,y,\omega)=d\delta_{T(x,\omega)}(y) \,d\leb(x) \,d\P (
\omega)
\]
by means of a measurable map
\[
T\dvtx\R^d\times\Omega\to\R^d\cup\{\eth\},
\]
defined uniquely almost everywhere.
Similarly, for each $z\in Z^d$ and $k\in\N$, there exists a
measurable map
\[
T_{z,k}\dvtx\R^d\times\Omega\times\Gamma\to
\R^d\cup\{\eth\}
\]
such that for each $\gamma\in\Gamma$
the measure
\[
dQ_{B_k(z,\gamma)}(x,y,\omega)=d\delta_{T_{z,k}(x,\omega,\gamma
)}(y) \,d\leb(x) \,d\P(
\omega)
\]
on
$\R^d\times\R^d\times\Omega$ is the unique optimal semicoupling between
$d\leb(x)$ and $1_{B_k(z,\gamma)}(y) \,d\mu^\omega(y) \,d\P(\omega)$.

%
%pr4.6 #&#
\begin{proposition}\label{quenchedmapconv}
For every $z\in Z^d$,
\[
T_{z,k}(x,\omega,\gamma) \to T(x,\omega) \qquad\mbox{as } k\to\infty \mbox
{ locally in } \leb\otimes\P\otimes\nu\mbox{-measure}.
\]
\end{proposition}

The claim basically relies on the following lemma
which is a slight modification (and extension) of a result in \cite
{Ambrosio-ln-ot}.

%
%le4.7 #&#
\begin{lemma}\label{ae-convergence}
Let $X, Y$ be locally compact Polish spaces,
$\theta$ a Radon measure on $X$
and $\rho$ a metric on $Y$ compatible with the topology.
\begin{longlist}[(ii)]
\item[(i)] For all $n\in\N$ let $T_n,T\dvtx X\to Y$ be Borel measurable
maps. Put\break
$dQ_n(x,y):=d\delta_{T_n(x)}(y)\,d\theta(x)$ and $dQ(x,y):=d\delta_{T(x)}(y)\,d\theta(x)$. Then
\[
T_n\to T \mbox{ locally in measure on }X \quad\Longleftrightarrow\quad
Q_n\to Q \mbox{ vaguely in }\mathcal M(X\times Y).
\]

\item[(ii)] More generally, let $T$ and $Q$ be as before whereas
\[
dQ_n(x,y):=\int_{X'}\,d\delta_{T_n(x,x')}(y)
\,d\theta' \bigl(x' \bigr) \,d\theta(x)
\]
for some probability space
$(X',\frak A',\theta')$ and
suitable measurable maps $T_n\dvtx X\times X'\to Y$. Then
\begin{eqnarray*}
&&Q_n\to Q \mbox{ vaguely in }\mathcal M(X\times Y)
\\
&&\qquad\Longrightarrow \quad T_n \bigl(x,x' \bigr)\to T(x) \mbox{
locally in measure on }X\times X'.
\end{eqnarray*}
\end{longlist}
\end{lemma}
\begin{pf}
(i)
Assume $T_n\to T$ in $\theta$-measure. Then also $f\circ(\mathrm{id},T_n)\to
f\circ(\mathrm{id},T)$ in $\theta$-measure for any $f\in C_c(X\times Y)$.
Therefore, by the dominated convergence theorem we have
\[
\int f(x,y)\,dQ_n=\int f \bigl(x,T_n(x) \bigr)\,d\theta
\to\int f \bigl(x,T(x) \bigr)\,d\theta=\int f(x,y)\,dQ.
\]
This proves the vague convergence of $Q_n$ toward Q.

For the opposite direction, fix $\tilde K\subset X$ compact and
$\varepsilon>0$. By Lusin's theorem there is a compact set $K\subset
\tilde K$ such that $T|_K$ is continuous and $\theta(\tilde K\setminus
K)<\varepsilon$. Put $\eta\dvtx\R_+\to\R_+, t\mapsto1\wedge\abs
{t}/\varepsilon$. The function
\[
\phi(x,y)=1_K(x)\eta \bigl(\rho \bigl(y,T(x) \bigr) \bigr)\vadjust{\goodbreak}
\]
is upper semicontinuous, nonnegative and compactly supported. Thus,
there exist $\phi_l\in C_c(X\times Y)$ with
$\phi_l\searrow\phi$.
By assumption, we have for each $l$
\[
\int\phi(x,y)\,dQ_n(x,y)\le\int\phi_l(x,y)
\,dQ_n(x,y) \stackrel{n\to\infty}\to\int\phi_l(x,y)
\,dQ(x,y).
\]
Moreover,
\[
\int\phi_l(x,y)\,dQ(x,y)\stackrel{l\to\infty}\to\int \phi(x,y)
\,dQ(x,y)=0.
\]
Therefore, $ \lim_{n\to\infty}\int\phi(x,y)\,dQ_n(x,y)=0$. In other words,
\[
\lim_{n\to\infty} \int1_K(x)\eta \bigl(\rho
\bigl(T_n(x),T(x) \bigr) \bigr)\,d\theta(x) = 0.
\]
This implies
$\lim_{n\to\infty}\theta(\{x\in K \dvtx\rho(T_n(x),T(x))\geq
\varepsilon\})=0$
and then in turn
\[
\lim_{n\to\infty}\theta \bigl( \bigl\{x\in\tilde K \dvtx\rho
\bigl(T_n(x),T(x) \bigr)\geq2\varepsilon \bigr\} \bigr)=0.
\]

(ii) Given any compact $\tilde K\subset X$ and any $\varepsilon>0$,
choose $\phi$ as before. Then vague convergence again implies
$ \lim_{n\to\infty}\int\phi(x,y)\,dQ_n(x,y)=0$. This, in other words,
now reads as
\[
\lim_{n\to\infty} \int_X\int_{X'}
1_K(x)\eta \bigl(\rho \bigl(T_n \bigl(x,x'
\bigr),T(x) \bigr) \bigr) \,d\theta' \bigl(x' \bigr)
\,d \theta(x) = 0.
\]
Therefore,
\[
\lim_{n\to\infty} \bigl(\theta\otimes\theta' \bigr) \bigl(
\bigl\{ \bigl(x,x' \bigr)\in\tilde K\times X' \dvtx
\rho \bigl(T_n \bigl(x,x' \bigr),T(x) \bigr)\geq2
\varepsilon \bigr\} \bigr)=0.
\]
This is the claim.
\end{pf}
\begin{pf*}{Proof of Proposition~\ref{quenchedmapconv}}
Fix $z\in Z^d$ and recall that
\[
Q_z^k\to Q^\infty\qquad\mbox{vaguely on }
\R^d\times\R^d,
\]
where
\[
dQ^\infty(x,y,\omega) = d\delta_{T(x,\omega)}(y) \,d\leb(x) \,d\P (
\omega)
\]
and
\begin{eqnarray*}
dQ_z^k(x,y,\omega) &= &\int_\Gamma
\,dQ_{B_k(z,\gamma)}(x,y,\omega) \,d\nu(\gamma)
\\
&= &\int_\Gamma \,d\delta_{T_{z,k}(x,\omega,\gamma)}(y) \,d\leb(x) \,d\P(
\omega) \,d\nu(\gamma)
\end{eqnarray*}
with transport maps $T\dvtx\R^d\times\Omega\to\R^d\cup\{\eth\}$ and
$T_{z,k}\dvtx\R^d\times\Omega\times\Gamma\to\R^d\cup\{\eth\}$
as above.
Apply assertion (ii) of the previous lemma
with
$X:=\R^d\times\Omega, X'=\Gamma, Y=\R^d\cup\{\eth\}$ and $\theta
=\leb
\otimes\P, \theta'=\nu$.
\end{pf*}

Actually, this convergence result can significantly be improved.

%
%th4.8 #&#
\begin{theorem}\label{convofmaps}
For every $z\in Z^d$ and every bounded Borel set $M\subset\R^d$,
\[
\lim_{k\to\infty} (\leb\otimes\P\otimes\nu) \bigl( \bigl\{ (x,\omega,\gamma)
\in M\times\Omega\times\Gamma\dvtx T_{z,k}(x,\omega,\gamma) \not = T(x,
\omega) \bigr\} \bigr) = 0.
\]
\end{theorem}

\begin{pf}
Let $M$ as above and $\varepsilon>0$ be given. Finiteness of the
asymptotic mean transportation cost implies that
there exists a bounded set $M'\subset\R^d$ such that
\[
(\leb\otimes\P) \bigl( \bigl\{(x,\omega)\in M\times\Omega\dvtx T(x,\omega)
\notin M' \bigr\} \bigr)\le\varepsilon.
\]
Given the bounded set $M'$
there exists $\delta>0$ such that the probability to find two distinct
particles of the point process at distance $<\delta$, at least one of
them within~$M'$, is less than $\varepsilon$, that is,
\begin{eqnarray*}
\P \bigl( \bigl\{\omega\dvtx\exists \bigl(y,y' \bigr)\in
M'\times\R^d\dvtx0<\bigl|y-y'\bigr|<\delta,
\mu^\omega \bigl(\{y\} \bigr)>0, \mu^\omega \bigl( \bigl
\{y' \bigr\} \bigr)>0 \bigr\} \bigr)
\le\varepsilon.
\end{eqnarray*}
On the other hand, Proposition~\ref{quenchedmapconv} states that with
high probability the maps $T$ and $T_{z,k}$ have distance less than
$\delta$. More precisely, for each $\delta>0$ there exists $k_0$ such
that for all $k\ge k_0$,
\[
(\leb\otimes\P\otimes\nu) \bigl( \bigl\{(x,\omega,\gamma)\in M\times\Omega
\times\Gamma\dvtx \bigl\llvert T_{z,k}(x,\omega,\gamma)- T(x,\omega)
\bigr\rrvert\ge\delta \bigr\} \bigr)\le\varepsilon.
\]
Since all the maps $T$ and $T_{z,k}$ take values in the support of the
point process (plus the point $\eth$) it follows
that
\[
(\leb\otimes\P\otimes\nu) \bigl( \bigl\{(x,\omega,\gamma)\in M\times\Omega
\times\Gamma\dvtx T_{z,k}(x,\omega,\gamma) \not= T(x,\omega) \bigr\}
\bigr)\le3\varepsilon
\]
for all $k\ge k_0$.
\end{pf}

%
%co4.9 #&#
\begin{corollary} There exists a subsequence $(k_l)_l$ such that
\[
T_{z,k_l}(x,\omega,\gamma) \to T(x,\omega) \qquad\mbox{as } l\to\infty
\]
for almost every $x\in\R^d$, $\omega\in\Omega$, $\gamma\in\Gamma
$ and
every $z\in Z^d$.
Indeed, the sequence $(T_{z,k_l})_l$ is finally stationary. That is,
there exists a random variable $l_z\dvtx\R^d\times\Omega\times
\Gamma\to\N$
such that almost surely
\[
T_{z,k_l}(x,\omega,\gamma) = T(x,\omega) \qquad\mbox{for all } l\ge
l_z(x,\omega,\gamma).
\]
\end{corollary}

%
%co4.10 #&#
\begin{corollary}
There is a measurable map $\Upsilon\dvtx\mathcal M(\R^d)\to\mathcal
M(\R^d\times\R^d)$ s.t. $q^\omega:=\Upsilon(\mu^\omega)$ denotes
the unique
optimal semicoupling between $\leb$ and $\mu^\omega$. In particular the
optimal semicoupling is a factor coupling.
\end{corollary}
\begin{pf}
By Theorem~\ref{euQ+q}, the maps $T_{z,k}$ are measurable with respect
to the sigma algebra generated by $\mu^\bullet$. By Theorem \ref
{convofmaps}, the optimal transportation map~$T$ is also measurable with
respect to the sigma algebra generated by $\mu^\bullet$. Because the
optimal semicoupling $q^\bullet$ is given by $q^\omega=(\mathrm{id},T^\omega
)_*\leb$, it is also measurable with respect to the sigma algebra
generated by $\mu^\bullet$. Thus there is a measurable map $\Upsilon$
such that $q^\bullet=\Upsilon(\mu^\bullet)$.
\end{pf}

%s5 #&#
\section{Estimates for the asymptotic mean transportation cost of a
Poisson process}\label{costsection}
Throughout this section, $\mu^\bullet$ will be a Poisson point process
of intensity $\beta\le1$.
The asymptotic mean transportation cost for $\mu^\bullet$ will be
denoted by
\[
\frak c_\infty=\frak c_\infty(\vartheta,d,\beta)
\]
or, if $\vartheta(r)=r^p$, by $\frak c_\infty(p,d,\beta)$. We will
present sufficient as well as necessary conditions for finiteness of
$\frak c_\infty$. These criteria will be quite sharp. Moreover, in the
case of $L^p$-cost, we also
present explicit sharp estimates for $\frak c_\infty$.

To begin with, let us summarize some elementary monotonicity properties
of $\frak c_\infty(\vartheta,d,\beta)$.

%
%le5.1 #&#
\begin{lemma}
\textup{(i)} $\vartheta\le\overline\vartheta$ implies $\frak
c_\infty
(\vartheta,d,\beta)\le\frak c_\infty(\overline\vartheta,d,\beta)$.

More generally, $\limsup_{r\to\infty}\frac{\overline\vartheta
(r)}{\vartheta(r)}<\infty$ and $\frak c_\infty(\vartheta,d,\beta
)<\infty
$ imply
$\frak c_\infty(\overline\vartheta,d,\break\beta)<\infty$.\vspace*{-6pt}
\begin{longlist}[(iii)]
\item[(ii)] If $\overline\vartheta=\varphi\circ\vartheta$ for some
convex increasing $\varphi\dvtx\R_+\to\R_+$, then
\[
\varphi \bigl(\beta^{-1}\frak c_\infty(\vartheta,d,\beta)
\bigr)\le\beta^{-1}\frak c_\infty(\overline\vartheta,d,\beta).
\]

\item[(iii)] $\beta\le\overline\beta$ implies $\frak c_\infty
(\vartheta
,d,\beta)\le\frak c_\infty(\vartheta,d,\overline{\beta})$.
\end{longlist}
\end{lemma}

\begin{pf}
(i) Is obvious. (ii) If $\overline q$ denotes the optimal semicoupling
for $\overline\vartheta$, then Jensen's inequality implies
\begin{eqnarray*}
&&\beta^{-1}\frak c_\infty(\overline\vartheta,d,\beta)
\\
&&\qquad=\beta^{-1}\EE\int_{\R^d\times[0,1)^d}\varphi \bigl(\vartheta
\bigl(|x-y|\bigr) \bigr) \,d\overline q(x,y)
\\
&&\qquad\ge\varphi \biggl(\beta^{-1}\EE\int_{\R^d\times[0,1)^d}
\vartheta\bigl(|x-y|\bigr) \,d\overline q(x,y) \biggr)\ge\varphi \bigl(\beta^{-1}
\frak c_\infty(\vartheta,d,\beta) \bigr).
\end{eqnarray*}

(iii) Given a realization $\overline\mu^\omega$ of a Poisson point
process with intensity $\overline\beta$. Delete each point $\xi\in
\supp
[\overline\mu^\omega]$ with probability $1- \beta/\overline\beta$,
independently of each other.
Then the remaining point process $\mu^\omega$ is a Poisson point
process with intensity~$\beta$. Hence, each semicoupling $\overline
q^\omega$ between $\leb$ and $\overline\mu^\omega$ leads to a
semicoupling $q^\omega$
between $\leb$ and $\mu^\omega$ with less or equal transportation cost.
The centers which survive are coupled with the same cells as before.
\end{pf}

%s5.1 #&#
\subsection{Lower estimates}

%
%th5.2 #&#
\begin{theorem}[(\cite{holroyd2001find})]\label{HP-d12}
Assume $\beta=1$ and $d\le2$. Then for all translation invariant
couplings of Lebesgue and Poisson
\[
\EE \biggl[\int_{\R^d\times[0,1)^d} |x-y|^{d/2}
\,dq^\bullet(x,y) \biggr]=\infty.
\]
\end{theorem}

%
%th5.3 #&#
\begin{theorem} For all $\beta\le1$ and $d\ge1$ there exists a
constant $\kappa'=\kappa'(d,\beta)$ such that
for all translation invariant semicouplings of Lebesgue and Poisson
\[
\EE \biggl[\int_{\R^d\times[0,1)^d} \exp \bigl(\kappa'|x-y|^{d}
\bigr) \,dq^\bullet(x,y) \biggr]=\infty.
\]
\end{theorem}

The result is well known in the case $\beta=1$. In this case, it is
based on a lower bound for the event ``no Poisson particle in the cube
$[-r,r)^d$'' and on a lower estimate for the cost of transporting the
Lebesgue measure
in $[-r/2,r/2)^d$ to some distribution on $\R^d\setminus[-r,r)^d$,
\[
\frak c_\infty\ge\exp \bigl(-(2r)^d \bigr)\cdot\vartheta
\biggl(\frac r2 \biggr)\cdot2^{-d}.
\]
Hence, $\frak c_\infty\to\infty$ as $r\to\infty$ if $\vartheta
(r)=\exp
(\kappa' r^d)$ with $\kappa'>2^{2d}$.

However, this argument breaks down in the case $\beta<1$. We will
present a different argument which works for all $\beta\le1$.

\begin{pf}
Consider the event ``more than $(3r)^d$ Poisson particles in the box
$[-r/2,r/2)^d$'' or, formally,
\[
\Omega(r)= \bigl\{\mu^\bullet \bigl([-r/2,r/2 )^d \bigr)
\ge(3r)^d \bigr\}.
\]
Note that $\EE\mu^\bullet([-r/2,r/2)^d )=\beta r^d$ with
$\beta\le1$.
For $\omega\in\Omega(r)$, the cost of a semicoupling between $\leb$ and
$1_{[-r/2,r/2)^d}\mu^\omega$ is bounded from below by
\[
\vartheta(r/2)\cdot r^d
\]
(since $r^d$ Poisson points---or more---must be transported at least
a distance $r/2$).
The large deviation result formulated in the next lemma allows us to estimate
\[
\mathbb P \bigl(\Omega(r_n) \bigr)\ge e^{-k\cdot{r_n}^d}
\]
for any $k> I_\beta(3^d)$ and suitable $r_n\to\infty$.
Hence,
if $\vartheta(r)\ge\exp(\kappa' r^d)$ with $\kappa'> 2^d\cdot k$, then
\[
\frak c_\infty\ge\mathbb P \bigl(\Omega(r_n) \bigr)\cdot
\vartheta(r/2)\ge\exp \bigl( \bigl(\kappa'2^{-d}-k
\bigr)r^d \bigr)\to\infty
\]
as $r\to\infty$.
\end{pf}

%
%le5.4 #&#
\begin{lemma}
Given any nested sequence of boxes $B_n(z,\gamma)\subset\R^d$ and
$t\geq\beta$
\[
\lim_{n\to\infty}\frac{-1}{2^{nd}} \log\mathbb P \biggl[\frac1{2^{nd}}
\mu^\bullet \bigl(B_n(z,\gamma) \bigr)\ge t \biggr] =
I_\beta(t)
\]
with $I_\beta(t)=t\log(t/\beta)-t+\beta$.
\end{lemma}

\begin{pf}
For a fixed sequence $B_n(z,\gamma)$, $n\in\N$, consider the sequence
of random variables $Z_n(\cdot)=\mu^\bullet(B_n(z,\gamma))$.
For each $n\in\N$,
\[
Z_n=\sum_{i\in B_n(z,\gamma)\cap\Z^d} X_i
\]
with $X_i=\mu^\bullet(B_0(i))$. The $X_i$ are i.i.d. Poisson random
variables with mean $\beta$.
Hence, Cram\'er's theorem states that for all $t\ge\beta$,
\[
\liminf_{n\to\infty}\frac{-1}{2^{nd}} \log\mathbb P \biggl[\frac
1{2^{nd}} Z_n\ge t \biggr]\ge I_\beta(t)
\]
with
\[
I_\beta(t)=\sup_x \bigl[tx-\log\hat\mu(x) \bigr]=t
\log(t/\beta)-t+\beta.
\]
\upqed\end{pf}

%s5.2 #&#
\subsection{Upper estimates for concave cost}

In this section we treat the case of a concave scale function
$\vartheta
$. In particular this implies that the cost function $c(x,y)=\vartheta
(|x-y|)$ defines a metric on $\R^d$. The results of this section will
be mainly of interest in the case $d\le2$; in particular, they will
prove assertion (ii) of Theorem~\ref{costthm}.
It suffices to consider the case $\beta=1$.
Similar to the early work of Ajtai, Koml\'os and Tusn\'ady
\cite{Ajtai-K-T}, our approach will be based on iterated transports
between cuboids of doubled edge length.

We put
%
%
%e7 #&#
\begin{equation}
\label{varTheta} \varTheta(r):=\int_{0}^r
\vartheta(s)\,\d s\quad \mbox{and}\quad \varepsilon(r):=\sup_{s\ge r}
\frac{\vartheta(s)}{s^{d/2}}.
\end{equation}

%s5.2.1 #&#
\subsubsection{Modified cost}

In order to prove the finiteness of the asymptotic mean transportation
cost, we will estimate the cost of a semicoupling between $\leb$ and
$1_A\mu^\bullet$ from above in terms of the cost of another, related coupling.

Given two measure-valued random variables $\nu_1^\bullet, \nu_2^\bullet
\dvtx\Omega\to\mathcal M(\R^d)$ with\break   $\nu_1^\omega(\R^d)=\nu_2^\omega(\R^d)$ for a.e. $\omega\in\Omega$, we define their
transportation distance
by
\[
\W_\vartheta(\nu_1,\nu_2):= \int
_\Omega W_\vartheta \bigl(\nu_1^\omega,
\nu_2^\omega \bigr) \,d\P(\omega),
\]
where
\[
W_\vartheta(\eta_1,\eta_2)=\inf \biggl\{ \int
_{\R^d\times\R^d}\vartheta\bigl(|x-y|\bigr) \,dq(x,y)\dvtx q \mbox{ is coupling of
} \eta_1,\eta_2 \biggr\}
\]
denotes the usual $L^1$-Wasserstein distance---w.r.t. the distance
$\vartheta(|x-y|)$---between (not necessarily normalized)
measures $\eta_1,\eta_2\in\mathcal M(\R^d)$ of equal total
mass.

%
%le5.5 #&#
\begin{lemma}\label{newWasserstein}
\textup{(i)} For any triple of random measures $\nu_1^\bullet, \nu_2^\bullet
,\nu_3^\bullet\dvtx\Omega\to\mathcal M(\R^d)$ with $\nu_1^\omega
(\R^d)=\nu_2^\omega(\R^d)=\nu_3^\omega(\R^d)$ for a.e. $\omega
\in\Omega,$ we have
the triangle inequality
\[
\W_\vartheta(\nu_1,\nu_3)\leq\W_\vartheta(
\nu_1,\nu_2)+\W_\vartheta(\nu_2,
\nu_3).
\]

\textup{(ii)} For each countable family of pairs of measure-valued random
variables $\nu_{1,k}^\bullet, \nu_{2,k}^\bullet\dvtx\Omega\to
\mathcal M(\R^d)$ with $\nu_{1,k}^\omega(\R^d)=\nu_{2,k}^\omega
(\R^d)$ for a.e.
$\omega\in\Omega$ and all $k$ we have
\[
\W_\vartheta \biggl(\sum_k
\nu_{1,k}^\bullet, \sum_k
\nu_{2,k}^\bullet \biggr)\le\sum_k
\W_\vartheta \bigl(\nu_{1,k}^\bullet,
\nu_{2,k}^\bullet \bigr).
\]
\end{lemma}
\begin{pf}
Gluing lemma (cf.~\cite{dudley2002real} or~\cite{villani2009optimal},
Chapter 1) plus Minkowski inequality yield (i); (ii) is obvious.
\end{pf}

For each bounded measurable $A\subset\R^d$ let us now define a random
measure $\nu_A^\bullet\dvtx \Omega\to\mathcal M(\R^d)$ by
\[
\nu_A^\omega\dvtx=\frac{\mu^\omega(A)}{\leb(A)}\cdot1_A
\leb.
\]
Note that---by construction---the measures $\nu_A^\omega$ and $1_A
\mu^\omega$ have the same total mass.
The \emph{modified transportation cost} is defined as
\begin{eqnarray*}
\widehat{\CCo}_A(\omega) &=& \inf \biggl\{\int c(x,y) \,\d\widehat
q(x,y) \dvtx\widehat q\mbox{ is coupling of $\nu_A^\omega$
and $1_A \mu^\omega$} \biggr\}
\\
&=& W_\vartheta \bigl(\nu_A^\omega,1_A
\mu^\omega \bigr).
\end{eqnarray*}
Put
\[
\widehat{\mathfrak c}_n = 2^{-nd}\cdot\EE[\widehat{
\CCo}_{B_n} ]
\]
with $B_n=[0,2^n)^d$ as usual.

%s5.2.2 #&#
\subsubsection{Semi-subadditivity of modified cost}

The crucial advantage of this modified cost function $\widehat{\CCo}_A$
is that it is semi-subadditive (i.e., subadditive up to correction
terms) on suitable classes of \emph{cuboids} which we are going to
introduce now.
For $n\in\N_0, k\in\{1,\ldots,d\}$ and $i\in\{0,1\}^k$, put
\[
B^i_{n+1}:=\bigl[0,2^n\bigr)^k
\times\bigl[0,2^{n+1}\bigr)^{d-k}+2^n\cdot(i_1,
\ldots,i_k,0,\ldots,0).
\]

These cuboids can be constructed by iterated subdivision of the
standard cube $B_{n+1}$ as follows:
We start with $B_{n+1}=[0,2^{n+1})^d$ and subdivide it (along the first
coordinate) into two disjoint congruent pieces
$B_{n+1}^{(0)}=[0,2^n)\times[0,2^{n+1})^{d-1}$ and
$B_{n+1}^{(1)}=B_{n+1}^{(0)}+2^n\cdot(1,0,\ldots,0)$.
In the $k$th step, we subdivide each of the
$B_{n+1}^i=B_{n+1}^{(i_1,\ldots,i_{k-1})}$ for $i\in\{0,1\}^{k-1}$
along the $k$th coordinate into two disjoint congruent pieces
$B_{n+1}^{(i_1,\ldots,i_{k-1},0)}$ and $B_{n+1}^{(i_1,\ldots,i_{k-1},1)}$.
After $d$ steps we are done. Each of the $B_{n+1}^i$ for $i\in\{0,1\}^d$ is a copy of the standard cube $B_n$, more precisely,
\[
B_{n+1}^i=B_n+2^n\cdot i.
\]

%
%le5.6 #&#
\begin{lemma}\label{boxmerger}
Given $n\in\N_0, k\in\{1,\ldots,d\}$ and $i\in\{0,1\}^k$ put
$D_0=\break B_{n+1}^{(i_1,\ldots,i_{k-1},0)}, D_1=B_{n+1}^{(i_1,\ldots
,i_{k-1},1)}$ and
$D=D_0\cup D_1=B_{n+1}^{(i_1,\ldots,i_{k-1})}$. Then
\[
\W_\vartheta(\nu_{D_0}+\nu_{D_1}, \nu_D
)\le2^{-(n+1)} \varTheta \bigl(2^{n+1} \bigr) 2^{d/2(n+1)-k/2},
\]
with $\varTheta$ as defined in \eqref{varTheta}.
\end{lemma}

\begin{pf}
Put $Z_j(\omega):=\mu^\omega(D_j)$ for $j\in\{0,1\}$. Then $Z_0,Z_1$
are independent Poisson random variables with parameter $\alpha_0=\alpha_1=\leb(D_j)=2^{d(n+1)-k}$, and $Z:=\mu(D)=Z_0+Z_1$ is a
Poisson random
variable with parameter $\alpha=2^{d(n+1)-k+1}$.

The measure $\nu_D$ has density $\frac{Z}\alpha$ on $D$ whereas the
measure $\tilde\nu_D:=\nu_{D_0}+\nu_{D_1}$ has density $\frac
{2Z_0}\alpha$ on the part $D_0\subset D$ and it has density $\frac
{2Z_1}\alpha$ on the remaining part $D_1\subset D$. If $Z=0$ nothing
has to be transported since $\tilde\nu$ already coincides with $\nu$.
Hence, for the sequel we may assume $Z>0$.

Assume that $Z_0>Z_1$. Then a total amount of mass $\frac{Z_0-Z_1}2$,
uniformly distributed over $D_0$, will be transported with the map
\[
T\dvtx(x_1,\ldots,x_{k-1},x_k,x_{k+1},
\ldots,x_d)\mapsto \bigl(x_1,\ldots,x_{k-1},2^{n+1}-
x_k,x_{k+1},\ldots,x_d \bigr)
\]
from $D_0$ to $D_1$. The rest of the mass remains where it is.
Hence, the cost of this transport is
\[
\frac{|Z_0-Z_1|}2 \cdot2^{-n}\int_0^{2^n}
\vartheta \bigl(2^{n+1}-2x_k \bigr) \,\d x_k=2^{-(n+2)}
\varTheta \bigl(2^{n+1} \bigr)\cdot|Z_0-Z_1|.
\]
Hence, we get
\begin{eqnarray*}
\W_\vartheta(\tilde\nu_D,\nu_D )&=&
2^{-(n+2)} \varTheta \bigl(2^{n+1} \bigr)\cdot\EE
\bigl[|Z_0-Z_1| \bigr]
\\
&\leq& 2^{-(n+1)} \varTheta \bigl(2^{n+1} \bigr)\cdot\EE
\bigl[|Z_0-\alpha_0| \bigr]
\\
&\leq& 2^{-(n+1)} \varTheta \bigl(2^{n+1} \bigr)\cdot
\alpha_0^{1/2}= 2^{-(n+1)} \varTheta
\bigl(2^{n+1} \bigr) 2^{d/2(n+1)-k/2}.
\end{eqnarray*}
\upqed\end{pf}

%
%pr5.7 #&#
\begin{proposition}\label{2thm3}
For all $n\in\N$ and arbitrary dimension $d$ the following holds:
\[
\widehat{\mathfrak c}_{n+1} \leq\widehat{\mathfrak c}_n+
2^{d/2+1}\cdot2^{-(n+1)(d/2+1)}\varTheta \bigl(2^{n+1} \bigr).
\]
\end{proposition}

\begin{pf}
By definition
\[
\W_\vartheta(1_{B_{n+1}}\mu, \nu_{B_{n+1}} )=2^{d(n+1)}
\cdot\widehat{\mathfrak c}_{n+1},\vadjust{\goodbreak}
\]
and it is easily observed that
\begin{eqnarray*}
\W_\vartheta \biggl(1_{B_{n+1}}\mu, \sum
_{i\in\{0,1\}^d}\nu_{B^i_{n}} \biggr)&\le& \sum
_{i\in\{0,1\}^d}\W_\vartheta(1_{B^i_{n}}\mu,
\nu_{B^i_{n}} )
\\
&=&2^d\cdot\W_\vartheta(1_{B_{n}}\mu,
\nu_{B_{n}} ) =2^{d(n+1)}\cdot\widehat{\mathfrak c}_{n}.
\end{eqnarray*}
Hence, by the triangle inequality for $\W_\vartheta$ an upper estimate
for $\widehat{\mathfrak c}_{n+1} - \widehat{\mathfrak c}_n$ will
follow from an upper bound for $\W_\vartheta(\sum_{i\in\{0,1\}
^d}\nu_{B^i_{n}}, \nu_{B_{n+1}} )$.

In order to estimate the cost of transportation from $\nu_{(d)}:=\sum_{i\in\{0,1\}^d}\nu_{B^i_{n}}$ to $\nu_{(0)}:=\nu_{B_{n+1}}$ for fixed
$n\in\N_0$, we introduce $(d-1)$ further (``intermediate'') measures
\[
\nu_{(k)}=\sum_{i\in\{0,1\}^k}\nu_{B_{n+1}^i}
\]
and estimate the cost of transportation from $\nu_{(k)}$ to $\nu_{(k-1)}$ for $k\in\{1,\ldots,d\}$.
For each $k$, these cost arise from \emph{merging $2^{k-1}$ pairs of
cuboids} into $2^{k-1}$ cuboids of twice the size. More precisely, from
moving mass within pairs of adjacent cuboids in order to obtain
equilibrium in the unified cuboid of twice the size. These costs---for
each of the $2^{k-1}$ pairs involved---have been estimated in the
previous lemma,
\begin{eqnarray*}
\W_\vartheta(\nu_{(k)},\nu_{(k-1)} )&\le&
2^{k-1}\cdot\W_\vartheta(\nu_{B_{n+1}^{i,0}}+\nu_{B_{n+1}^{i,1}},
\nu_{B_{n+1}^i} )
\\
&\le& 2^{k-1}\cdot2^{-(n+1)} \varTheta \bigl(2^{n+1}
\bigr) 2^{d/2(n+1)-k/2}
\end{eqnarray*}
for $k\in\{1,\ldots,d\}$ (and arbitrary $i\in\{0,1\}^{k-1}$). Thus
\begin{eqnarray*}
2^{d(n+1)}\cdot[\widehat{\mathfrak c}_{n+1}-\widehat{\mathfrak
c}_{n} ]&\le&\W_\vartheta(1_{B_{n+1}}\mu,
\nu_{(0)} )- \W_\vartheta(1_{B_{n+1}}\mu, \nu_{(d)}
)
\\
&\le&\sum_{k=1}^d\W_\vartheta(
\nu_{(k-1)}, \nu_{(k)} )
\\
&\le&\sum_{k=1}^d 2^{k/2}
\cdot2^{-(n+2)} \varTheta \bigl(2^{n+1} \bigr) 2^{d/2(n+1)}
\\
&\le& 4 \cdot2^{(n+2)(d/2-1)}\cdot\varTheta \bigl(2^{n+1} \bigr),
\end{eqnarray*}
which yields the claim.
\end{pf}

%
%co5.8 #&#
\begin{corollary}\label{sum-theta}
If $\sum_{n\geq1} 2^{-(n+1)(d/2+1)}\varTheta(2^{n+1})<\infty$, we have
\[
\widehat{\mathfrak c}_\infty:= \lim_{n\to\infty} \widehat {\mathfrak
c}_n
\]
exists and is finite.
\end{corollary}
\begin{pf}
According to the previous proposition,
%
%
%e8 #&#
\begin{equation}
\label{c-vs-hatc} \lim_{n\to\infty}\widehat{\mathfrak c}_n\leq
\widehat{\mathfrak c}_N + \sum_{m\geq N}
2^{-(m+1)(d/2+1)}\varTheta \bigl(2^{m+1} \bigr)
\end{equation}
for each $N\in\N$. As the sum was assumed to converge, the claim follows.
\end{pf}

%s5.2.3 #&#
\subsubsection{Comparison of costs}
Recall the definition of $\mathfrak c_n$ from Section~\ref{standardexhaustion}.

%
%pr5.9 #&#
\begin{proposition}\label{2thm4}
For all $d\in\N$ and for all $n\in\mathbb N_0$,
\[
{\mathfrak c}_{n} \leq\widehat{\mathfrak c}_n+ \sqrt{2d}
\cdot\varepsilon \bigl(2^{n} \bigr).
\]
\end{proposition}

\begin{pf}
Let a box $B=B_n=[0,2^n)^d$ for some fixed $n\in\N_0$ be given. We
define a measure-valued random variable
$\lambda_B^\bullet: \Omega\to\mathcal M(\R^d)$ by
\[
\lambda^\omega_B=1_{\widehat B(\omega)}\cdot\leb
\]
with a randomly scaled box $\widehat B(\omega)=[0,Z(\omega
)^{1/d})^d\subset\R^d$ and $Z(\omega)=\mu^\omega(B)$. Recall that $Z$
is a Poisson random variable with parameter $\alpha=2^{nd}$. Moreover,
note that
\[
\lambda^\omega_B \bigl(\R^d \bigr)=
\mu^\omega(B)=\nu^\omega_B \bigl(\R^d
\bigr)
\]
and that $\lambda^\omega_B\le\leb$
for each $\omega\in\Omega$. Each coupling of $\lambda^\omega_B$ of
$1_B\mu^\omega$, therefore, is also a semicoupling of $\leb$ and
$1_B\mu^\omega$. Hence,
\[
2^{nd}\cdot{\mathfrak c}_{n}\le\W_\vartheta(
\lambda_B, 1_B\mu).
\]
On the other hand, obviously,
\[
2^{nd}\cdot\widehat{\mathfrak c}_{n}= \W_\vartheta(
\nu_B, 1_B\mu)
\]
and thus
\[
2^{nd}\cdot({\mathfrak c}_{n} - \widehat{\mathfrak
c}_n )\le\W_\vartheta(\nu_B,\lambda_B).
\]

If $Z>\alpha$ a transport $T_*\nu_B=\lambda_B$ can be constructed as
follows: at each point of $B$ the portion $\frac\alpha Z$ of $\nu_B$
remains where it is; the rest is transported from $B$ into $\widehat
B\setminus B$. The maximal transportation distance is $\sqrt{d}\cdot
Z^{1/d}$. Hence, the cost can be estimated by
\[
\vartheta \bigl(\sqrt{d}\cdot Z^{1/d} \bigr)\cdot(Z-\alpha).
\]
On the other hand, if $Z<\alpha$ in a similar manner, a transport
$T'_*\lambda_B=\nu_B$ can be constructed with cost bounded from above by
\[
\vartheta \bigl(\sqrt{d}\cdot\alpha^{1/d} \bigr)\cdot(\alpha-Z).
\]
Therefore, by definition of the function $\varepsilon(\cdot)$,
\begin{eqnarray*}
\W_\vartheta(\nu_B,\lambda_B) &\le& \EE \bigl[
\vartheta \bigl(\sqrt{d}(Z\vee\alpha)^{1/d} \bigr)\cdot\llvert Z-\alpha
\rrvert \bigr]
\\
&\le& \varepsilon \bigl(\alpha^{1/d} \bigr)\cdot\sqrt{d}\cdot\EE
\bigl[(Z\vee\alpha)^{1/2}\cdot\llvert Z-\alpha\rrvert \bigr]
\\
&\le& \varepsilon \bigl(\alpha^{1/d} \bigr)\cdot\sqrt{d}\cdot\EE[Z+
\alpha]^{1/2}\cdot\EE \bigl[\llvert Z-\alpha\rrvert^2
\bigr]^{1/2}
\\
&=& \varepsilon \bigl(2^n \bigr)\cdot\sqrt{d}\cdot \bigl[2
\cdot2^{nd}\cdot2^{nd} \bigr]^{1/2}.
\end{eqnarray*}
This finally yields
\[
{\mathfrak c}_{n} - \widehat{\mathfrak c}_n \le
2^{-nd}\cdot\W_\vartheta(\nu_B,\lambda_B)
\le\varepsilon \bigl(2^{n} \bigr)\cdot\sqrt{2d}.
\]
\upqed\end{pf}

%
%th5.10 #&#
\begin{theorem}
Assume that
%
%
%e9 #&#
\begin{equation}
\label{integral-cond}\int_1^\infty
\frac{\vartheta
(r)}{r^{1+d/2}}\,dr<\infty
\end{equation}
then
\[
{\mathfrak c}_\infty\le\widehat{\mathfrak c}_\infty< \infty.
\]
\end{theorem}

\begin{pf} Since
\[
\int_1^\infty\frac{\vartheta(r)}{r^{1+d/2}}\,dr<\infty
\quad\Longleftrightarrow\quad\sum_{n=1}^\infty
\frac{\varTheta(2^n)}{2^{n(1+d/2)}}<\infty,
\]
Corollary~\ref{sum-theta}
applies and yields $\widehat{\mathfrak c}_\infty< \infty$.
Moreover, since
$\vartheta$ is increasing, the integrability condition (\ref
{integral-cond}) implies that
\[
\varepsilon(r)=\sup_{s\ge r}\frac{\vartheta(s)}{s^{d/2}} \to0
\]
as $r\to\infty$. Hence,
${\mathfrak c}_\infty\le\widehat{\mathfrak c}_\infty$
by Proposition~\ref{2thm4}.
\end{pf}

The previous theorem
essentially says that $\mathfrak c_\infty<\infty$ if $\vartheta$ grows
``slightly'' slower than $r^{d/2}$. This criterion is quite sharp in
dimensions 1 and 2.
Indeed, according to Theorem~\ref{HP-d12} in these two cases we also
know that $\mathfrak c_\infty=\infty$ if $\vartheta$ grows like
$r^{d/2}$ or faster.

%s5.3 #&#
\subsection{Estimates for $L^p$-cost}

The results of the previous section in particular apply to $L^p$-cost
for $p< d/2$ in $d\le2$ and to $L^p$-cost for $p\le1$ in $d\ge3$.
A~slight modification of these arguments will allow us
to deduce cost estimates for $L^p$ cost for arbitrary $p\ge1$ in the
case $d\geq3$.

In this case, the finiteness of $\mathfrak c_\infty$ will also be
covered by the more general results of~\cite{extra-heads}; see Theorem
\ref{costthm}(i). However, using the idea of modified cost we get reasonably
good quantitative estimates on $\mathfrak c_\infty$.
Throughout this section we assume $\beta=1$.

%s5.3.1 #&#
\subsubsection{Some moment estimates for Poisson random variables}

For $p\in\R$ let us denote by $\lceil p\rceil$
the smallest integer $\ge p$.

%
%le5.11 #&#
\begin{lemma}\label{thm2-app}\label{thm-app}
For each $p\in(0,\infty)$ there exist constants $C_1(p), C_2(p)$ and
$C_3(p)$ such that for every Poisson random variable $Z$ with parameter
$\alpha\geq1$:
\begin{longlist}[(iii)]%\label{Poi1}
\item[(i)] $\EE[Z^p ]\leq C_1(p)\cdot\alpha^p$, where one
can choose $C_1(1)=1, C_1(2)=4$.

For general $p$ one may choose $C_1(p)=\lceil p\rceil^p$ or
$C_1(p)=2^{p-1}\cdot(\lceil p\rceil-1)!$.
\item[(ii)] $\EE[Z^{-p}\cdot1_{\{Z>0\}} ]\leq C_2(p) \cdot
\alpha^{-p}$.

For general $p$ one may choose $C_2(p)=(\lceil p\rceil+1)!$.
\item[(iii)] $\EE[(Z-\alpha)^{p} ]\leq C_3(p)\cdot\alpha^{p/2}$,
where one can choose \mbox{$C_3(2)=1, C_1(4)=2$}.

For general $p$ one may choose $C_3=2^{p-1}\cdot(2\lceil\frac
p2\rceil-1)!$.
\end{longlist}
\end{lemma}
\begin{pf} In all cases, by H\"older's inequality it suffices to prove
the claim for integer $p\in\N$.

(i) The moment generating function of $Z$ is
\[
M(t):=\EE \bigl[e^{tZ} \bigr]=\exp \bigl(\alpha \bigl(e^t-1
\bigr) \bigr).
\]
For integer $p$, the $p$th moment of $Z$ is given by the $p$th
derivative of $M$ at the point $t=0$, that is, $\EE[Z^p
]=M^{(p)}(0)$. As a function of $\alpha$, the $p$th derivative of $M$
is a polynomial of order $p$ (with coefficients depending on $t$). As
$\alpha\geq1$ we are done.

To get quantitative estimates for $C_1$, observe that differentiating
$M(t)$ p times yields at most $2^{p-1}$ terms, each of them having a
coefficient $\le(p-1)!$ (if we do not merge terms of the same order).
Thus, we can take $C_1=2^{p-1}\cdot(p-1)!$.

Alternatively, we may use the
recursive formula
\[
T_{n+1}(\alpha)=\alpha\sum_{k=0}^n
\pmatrix{ n\cr k} T_k(\alpha)
\]
for the \emph{Touchard polynomials} $T_n(\alpha):=\EE[Z^n]$; see, for
example,~\cite{Touchard}.
Assuming that $T_k(\alpha)\le(k\alpha)^k$ for all $k=1,\ldots,n$ leads
to the corresponding estimate for $k=n+1$.

(iii) Put $p=2k$ with integer $k$. The moment generating function of
$(Z-\alpha)$ is
\begin{eqnarray*}
N(t)&:=&\exp \bigl(\alpha \bigl(e^t-1-t \bigr) \bigr)=\exp \biggl({
\frac\alpha2t^2h(t)} \biggr)
\\
&=&1+\frac\alpha2t^2h(t)+\frac12 \biggl(\frac\alpha2
\biggr)^2t^4h^2(t)+\frac16 \biggl(\frac\alpha2
\biggr)^3t^6h^3(t)+\cdots
\end{eqnarray*}
with $h(t)=\frac2{t^2}(e^t-1-t)$.
Hence, the $2k$th derivative of $N$ at the point $t=0$ is a polynomial
of order $k$ in $\alpha$.
Since $\alpha\ge1$ by assumption, $\EE[(Z-\alpha)^{2k}] =
N^{(2k)}(0)\le C_3\cdot\alpha^k$ for some $C_3$.
To estimate $C_3$, again observe that differentiating $N(t)$ (2k) times
yields at most $2^{2k-1}$ terms. Each of these terms has a coefficient
$\le(2k-1)!$ (if we do not merge terms). Hence we can take
$C_3(2k)=2^{2k-1}\cdot(2k-1)!$.

(ii) The result follows from the inequality
\[
\frac{1}{x^k}\leq\frac{(k+1)!x!}{(k+x)!}
\]
for positive integers $k$ and $x$. The inequality is equivalent to
\[
\pmatrix{x+k
\cr
x-1}\leq x^{k+1}.
\]
For fixed $k$ the latter inequality holds for $x=1$. If $x$ increases
from $x$ to $x+1$ the right-hand side grows by a factor of $ (\frac
{x+1}{x} )^{k+1}$ and the left-hand side by a factor of $\frac
{x+k+1}{x}$. As
$(x+k+1)x^k\leq(x+1)^{k+1}$,
the inequality holds. Then we can estimate
\begin{eqnarray*}
\EE \biggl[\frac{1}{Z^k}\cdot1_{Z>0} \biggr]&\leq& \EE \biggl[
\frac
{(k+1)!}{(Z+1)\cdots(Z+k)}\cdot1_{Z>0} \biggr]
\\
&=&e^{-\alpha}\cdot\sum_{j=1}^\infty
\frac{\alpha^j}{j!}\cdot\frac
{(k+1)!}{(j+1)\cdots(j+k)}
\\
&=& \frac{(k+1)!}{\alpha^k}\cdot e^{-\alpha}\cdot\sum
_{j=1}^\infty\frac{\alpha^{j+k}}{(j+k)!} \le
\frac{(k+1)!}{\alpha^k}.
\end{eqnarray*}
If we choose $k=\lceil p \rceil$, this yields the claim.
\end{pf}

%s5.3.2 #&#
\subsubsection{\texorpdfstring{$L^p$-cost for $p\ge1$ in $d\ge3$}
{Lp-cost for p>=1 in d>=3}}
Given two measure valued random variables $\nu_1^\bullet, \nu_2^\bullet
\dvtx\Omega\to\mathcal M(\R^d)$ with $\nu_1^\omega(\R^d)=\nu_2^\omega(\R^d)$ for a.e. $\omega\in\Omega$, we define their
$L^p$-transportation distance
by
\[
\W_p(\nu_1,\nu_2):= \biggl[\int
_\Omega W_p^{p} \bigl(
\nu_1^\omega,\nu_2^\omega \bigr) \,d\P(
\omega) \biggr]^{1/p},
\]
where
\[
W_p(\eta_1,\eta_2)=\inf \biggl\{ \biggl[\int
_{\R^d\times\R^d}|x-y|^p \,d\theta(x,y) \biggr]^{1/p}
\dvtx\theta\mbox{ is coupling of }\eta_1,\eta_2 \biggr\}
\]
denotes the usual $L^p$-Wasserstein distance between (not necessarily
normalized)
measures $\eta_1,\eta_2\in\mathcal M(\R^d)$ of equal total mass.
Note that $\W_p(\nu_1,\nu_2)$ is \emph{not} the $L^p$-Wasserstein
distance between the distributions of $\nu_1^\bullet$ and $\nu_2^\bullet$.
The latter in general is smaller.\vadjust{\goodbreak} Similar to the concave case the
triangle inequality holds, and we define the \emph{modified
transportation cost} as
\begin{eqnarray*}
\widehat{\CCo}_A(\omega) &=& \inf \biggl\{\int|x-y|^p
\,\d\widehat q(x,y) \dvtx\widehat q\mbox{ is coupling of $\nu_A^\omega$
and $1_A \mu^\omega$} \biggr\}
\\
&=& W_p^{p} \bigl(\nu_A^\omega,1_A
\mu^\omega \bigr).
\end{eqnarray*}
Put
\[
\widehat{\mathfrak c}_n = 2^{-nd}\cdot\EE[\widehat{\CCo
}_{B_n} ]=\W_p^{p} \bigl(\nu_{B_n}^\bullet,1_{B_n}
\mu^\bullet \bigr)
\]
with $B_n=[0,2^n)^d$ as usual.

%
%le5.12 #&#
\begin{lemma}
Given $n\in\N_0, k\in\{1,\ldots,d\}$ and $i\in\{0,1\}^k$ put
$D_0=B_{n+1}^{(i_1,\ldots,i_{k-1},0)}, D_1=B_{n+1}^{(i_1,\ldots
,i_{k-1},1)}$ and
$D=D_0\cup D_1=B_{n+1}^{(i_1,\ldots,i_{k-1})}$. Then for some constant
$\kappa_1$ depending only on $p$,
\[
\W_p^{p} (\nu_{D_0}+\nu_{D_1},
\nu_D )\le\kappa_1\cdot2^{(n+1)(p+d-pd/2)}
\cdot2^{k(p/2-1)+1}.
\]
One may choose $\kappa_1(p)=\frac1{p+1}2^{-p}\cdot C_3(2p)\cdot C_2(2(p-1))$.
\end{lemma}
\begin{pf} The proof will be a modification of the proof of Lemma \ref
{boxmerger}.
An optimal transport map $T\dvtx D\to D$ with $T_*\tilde\nu_D=\nu_D$
is now
given by
\[
T\dvtx(x_1,\ldots,x_{k-1},x_k,x_{k+1},
\ldots,x_d)\mapsto \biggl(x_1,\ldots,x_{k-1},
\frac{2Z_0}Z\cdot x_k,x_{k+1},\ldots,x_d
\biggr)
\]
on $D_0$ and
\begin{eqnarray*}
&&T\dvtx(x_1,\ldots,x_{k-1},x_k,x_{k+1},
\ldots,x_d)
\\
&&\qquad\mapsto \biggl(x_1,\ldots,x_{k-1},2^{n+1}-
\bigl(2^{n+1}-x_k \bigr)\cdot\frac
{2Z_1}Z,x_{k+1},
\ldots,x_d \biggr)
\end{eqnarray*}
on $D_1$. As before, we put $Z_j(\omega)=\mu^\omega(D_j)$ for $j=0,1$
and $Z=Z_0+Z_1.$
(If $p>1$ this is indeed the only optimal transport map.) The cost of
this transport can easily be calculated,
\begin{eqnarray*}
\int_{D_0}\bigl|T(x)-x\bigr|^p \,d\tilde
\nu(x)&=&Z_0\cdot2^{-n}\int_0^{2^n}
\biggl\llvert\frac{2Z_0}Z\cdot x_k-x_k \biggr
\rrvert^p \,dx_k
\\
&=&\frac{2^{np}}{p+1}\cdot Z_0\cdot \biggl\llvert
\frac{Z_0-Z_1}{Z} \biggr\rrvert^p
\end{eqnarray*}
and analogously
\[
\int_{D_1}\bigl|T(x)-x\bigr|^p \,d\tilde\nu(x)=
\frac{2^{np}}{p+1}\cdot Z_1\cdot \biggl\llvert\frac
{Z_0-Z_1}{Z}
\biggr\rrvert^p.
\]
Hence, together with the estimates from Lemma~\ref{thm-app} this yields
\begin{eqnarray*}
\W_p^p (\tilde\nu_D,\nu_D )
&=&\frac{2^{np}}{p+1}\cdot\EE \biggl[\frac
{|Z_0-Z_1|^p}{Z^{p-1}}\cdot1_{\{Z>0\}}
\biggr]
\\
&\le&\frac{2^{np}}{p+1}\cdot\EE \bigl[|Z_0-Z_1|^{2p}
\bigr]^{1/2}\cdot\EE \bigl[Z^{-2(p-1)}\cdot1_{\{Z>0\}}
\bigr]^{1/2}
\\
&\le&\frac{2^{(n+1)p}}{p+1}\cdot\EE \bigl[|Z_0-\alpha_0|^{2p}
\bigr]^{1/2}\cdot\EE \bigl[Z^{-2(p-1)}\cdot1_{\{Z>0\}}
\bigr]^{1/2}
\\
&\le&\frac{2^{(n+1)p}}{p+1}\cdot C_3\cdot\alpha_0^{p/2}
\cdot C_2\cdot\alpha^{1-p}
\\
&\le&\kappa_1\cdot2^{(n+1)(p+d-pd/2)}\cdot2^{k(p/2-1)+1},
\end{eqnarray*}
which is the claim.
\end{pf}
With the very same proof as before (Proposition~\ref{2thm3}), by
inserting different results, we get the following:
%
%
%pr5.13 #&#
\begin{proposition}
For all $d\in\N$ and all $p\geq1$, there is a constant $\kappa_2=\kappa_2(p,d)$ such that for all $n\in\mathbb N_0$,
\[
\widehat{\mathfrak c}_{n+1}^{1/p} \leq\widehat{\mathfrak
c}_n^{1/p}+ \kappa_2\cdot2^{(n+1)
(1-d/2)}.
\]
One may choose
$\kappa_2(p,d)=\kappa_1(p)^{1/p}\cdot\sum_{k=1}^d2^{k/2}\le\kappa_1(p)^{1/p}\cdot2^{d/2+2},$ where $\kappa_1$ is the constant from the
previous lemma.
\end{proposition}

%
%co5.14 #&#
\begin{corollary}\label{chutkvgz}
For all $d\geq3$ and all $p\geq1$,
\[
\widehat{\mathfrak c}_\infty:= \lim_{n\to\infty}\widehat {\mathfrak
c}_n < \infty.
\]
More precisely, for all $n\in\N_0$,
\[
\widehat{\mathfrak c}_\infty^{1/p} \leq\widehat{\mathfrak
c}_n^{1/p}+\kappa_2\cdot\frac{2^{-(n+1)(d/2-1)}}{1-2^{-(d/2-1)}}.
\]
In particular,
\[
\widehat{\mathfrak c}_\infty^{1/p} \leq\widehat{\mathfrak
c}_0^{1/p}+\frac{4\kappa_1(p)^{1/p}}{2^{-1}-2^{-d/2}}.
\]
\end{corollary}

Recall the definition of $\mathfrak c_n$ from Section \ref
{standardexhaustion}. Comparison of costs $\widehat{\mathfrak c}_n$ and
${\mathfrak c}_n$ now yields the following:

%
%pr5.15 #&#
\begin{proposition}
For all $d\geq3$ and all $p\geq1$, there is a constant $\kappa_3$
such that for all $n\in\mathbb N_0$,
\[
{\mathfrak c}_{n}^{1/p} \leq\widehat{\mathfrak
c}_n^{1/p}+ \kappa_3\cdot2^{n(1-d/2)}.
\]
\end{proposition}
\begin{pf}
It is a modification of the proof of Proposition~\ref{2thm4}. This time,
the map $T\dvtx B\mapsto\widehat B$
\[
T\dvtx x\mapsto \biggl(\frac Z \alpha \biggr)^{1/d}\cdot x
\]
defines an optimal transport $T_*\nu_B=\lambda_B$. Put $\tau'=\tau'(d,p)=\int_{[0,1)^d}|x|^p \,dx$. (This can easily be estimated,
e.g., by $\tau'\le\frac1{p+1}d^{p/2}$ if $p\ge2$.)
The cost of the transport $T$ is
\begin{eqnarray*}
\int_B\bigl|T(x)-x\bigr|^p \,d\nu_B(x)&=&
\tau'\cdot2^{np}\cdot Z\cdot \biggl\llvert \biggl(\frac Z
\alpha \biggr)^{1/d}-1 \biggr\rrvert^p
\\
&\le&\tau'\cdot2^{np}\cdot Z\cdot \biggl\llvert\frac Z
\alpha-1 \biggr\rrvert^p.
\end{eqnarray*}
The inequality in the above estimation follows from the fact that
$|t-1|\le|t-1|\cdot(t^{d-1}+\cdots+t+1|= |t^d-1|$ for each real $t>0$.
The previous cost estimates hold true for each fixed $\omega$ (which
for simplicity we had suppressed in the notation). Integrating w.r.t.
$d\mathbb P(\omega)$ yields
\begin{eqnarray*}
\W_p^p(\nu_B,\lambda_B) &
\le& \tau'\cdot2^{np}\cdot\EE \biggl[ Z\cdot \biggl\llvert
\frac Z \alpha-1 \biggr\rrvert^p \biggr]
\\
&\le& \tau'\cdot2^{np}\cdot\alpha^{-p}\cdot\EE
\bigl[ Z^2 \bigr]^{1/2}\cdot\EE \bigl[|Z-
\alpha|^{2p} \bigr]^{1/2}
\\
&\le& \tau'\cdot2^{np}\cdot\alpha^{-p} \cdot
\alpha\cdot C_3\cdot\alpha^{p/2} = \kappa_3^p
\cdot2^{n(d+p-dp/2)}
\end{eqnarray*}
and thus
\[
{\mathfrak{c}}_{n}^{1/p^{*}} - \widehat{\mathfrak{c}}_n^{1/p^*}
\le\kappa_3\cdot2^{n(1-d/2)}.
\]
\upqed
\end{pf}

%
%co5.16 #&#
\begin{corollary}
For all $d\geq3$ and all $p\geq1$,
\[
{\mathfrak c}_\infty\le\widehat{\mathfrak c}_\infty< \infty.
\]
\end{corollary}

%s5.3.3 #&#
\subsubsection{Quantitative estimates}
Throughout this section, we assume that $\vartheta(r)=r^p$ with
$p<\overline p(d)$ where
\[
p<\overline p(d):= \cases{
\infty,&\quad
$\mbox{for } d\ge3$,
\vspace*{2pt}\cr
1,&\quad $\mbox{for } d=2$,
\vspace*{2pt}\cr
\frac12,&\quad $\mbox{for } d=1.$}
\]

%
%pr5.17 #&#
\begin{proposition}\label{c0-vs-tau}
Put $\tau(p,d)=\frac d{d+p}\cdot(\Gamma(\frac d2+1)^{1/d}\cdot\pi^{-1/2} )^p$.
Then
\[
{\mathfrak c}_\infty\ge{\mathfrak c}_0\ge\tau(p,d).
\]
\end{proposition}

\begin{pf} The number $\tau$ as defined above is the minimal cost of a
semicoupling between $\leb$ and a single Dirac mass, say $\delta_0$.
Indeed, this Dirac mass will be transported onto the $d$-dimensional
ball $K_r=\{x\in\R^d\dvtx|x|<r\}$ of unit volume, that is, with
radius $r$
chosen s.t. $\leb(K_r)=1$.
The cost of this transport is $\int_{K_r}|x|^p \,dx=\frac d{d+p}r^p=\tau$.

For each integer $Z\ge2$, the minimal cost of a semicoupling between
$\leb$ and a sum of $Z$ Dirac masses will be $\ge Z\cdot\tau$.
Hence, if $Z$ is Poisson distributed with parameter $1$,
\[
{\mathfrak c}_0 \ge{\mathbb E}[Z]\cdot\tau=\tau.
\]
\upqed\end{pf}

%
%re5.18 #&#
\begin{remark}
Explicit calculations yield
\begin{eqnarray*}
\tau(p,1)&=&\frac1{1+p}\cdot2^{-p},\qquad \tau(p,2)=\frac2{2+p}\cdot{\pi
}^{-p/2},\\ \tau(p,3)&=&\frac3{3+p}\cdot \biggl(\frac3{4\pi}
\biggr)^{p/3}
\end{eqnarray*}
whereas Stirling's formula yields a uniform lower bound, valid for all
$d\in\N$ (which indeed is a quite good approximation for large $d$)
\[
\tau(p,d)\ge\frac d{d+p}\cdot \biggl(\frac{d}{2\pi e} \biggr)^{p/2}.
\]
\end{remark}

%
%pr5.19 #&#
\begin{proposition}\label{tauhutbounds}
Put
$\widehat\tau=\widehat\tau(d,p)=\int_{[0,1)^d}\int_{[0,1)^d}|x-y|^p\,
dy\,dx$. Then
\[
e^{-1}\cdot\widehat\tau\le\widehat{\mathfrak c}_0 \le
\widehat\tau.
\]
Moreover,
$\widehat\tau\le\frac1{(1+p)(1+p/2)}\cdot d^{p/2}$ for all $p\ge2$
and
$\widehat\tau\le(\frac d6 )^{p/2}$ for all $0<p\le2.$
\end{proposition}

\begin{pf} If there is exactly one Poisson particle in $B_0=[0,1)^d$---which then is uniformly distributed-- then the transportation cost is
exactly $\widehat\tau(d,p)$. If there are $N>1$ particles in $B_0$, the
cost per particle is by definition of $\widehat{\mathfrak c}_0$
bounded by $\widehat\tau(d,p)$. Hence, we can bound $\widehat
{\mathfrak c}_0$ by the expected number of particles in $B_0$ times
$\widehat\tau(d,p)$ which is precisely $\widehat\tau(d,p)$. The number
of particles will be Poisson distributed with parameter 1. The lower
estimate for the cost follows from the fact that with probability
$e^{-1}$ there is exactly one Poisson particle in $B_0=[0,1)^d$.

Using the inequality $(x_1^2+\cdots+x_d^2)^{p/2}\le d^{p/2-1}\cdot
(x_1^p+\cdots+x_d^p)$---valid for all $p\ge2$---the upper estimate
for $\widehat\tau$ can be derived as follows:
\begin{eqnarray*}
\int_{[0,1)^d}\int_{[0,1)^d}|x-y|^p
\,dy\,dx &\le&d^{p/2-1}\sum_{i=1}^d
\int_{[0,1]^d}\int_{[0,1]^d}|x_i-y_i|^p
\,dy\,dx
\\
&=&d^{p/2}\int_0^1\int
_0^1|s-t|^p \,ds\,dt
\\
&=&\frac1{(1+p) (1+p/2)}\cdot d^{p/2}.
\end{eqnarray*}

Applying H\"older's inequality to the inequality for $p=2$ yields the
claim for all $p\le2$.
\end{pf}

%
%th5.20 #&#
\begin{theorem}
For all $p\le1$ and $d>2p$,
\[
\frac d{d+p}\cdot \biggl(\frac{d}{2\pi e} \biggr)^{p/2} \le{
\mathfrak c}_\infty\le \biggl(\frac{d}6 \biggr)^{p/2}+
\frac1{(p+1) \bigl(2^{d/2-p}-1 \bigr)}
\]
whereas for all $p\ge1$ and $d\ge3$,
\[
\biggl(\frac d{d+p} \biggr)^{1/p}\cdot \biggl(\frac{d}{2\pi e}
\biggr)^{1/2} \le{\mathfrak c}_\infty^{1/p} \le
\frac{d^{1/2}}{6^{1/2}
\wedge[(1+p)(1+p/2)]^{1/p}}+28\cdot\kappa_1^{1/p}.
\]
\end{theorem}
\begin{pf}
Proposition~\ref{c0-vs-tau} and the subsequent remark imply the lower bound
\[
\frac d{d+p}\cdot \biggl(\frac{d}{2\pi e} \biggr)^{p/2} \le\tau\le {
\mathfrak c}_\infty,
\]
valid for all $d$ and $p$.
In the case $p\ge1$ the upper bound follows from Proposition~\ref
{tauhutbounds} and Corollary~\ref{chutkvgz} by
\[
{\mathfrak c}_\infty^{1/p} \le\widehat\tau^{1/p} +
\frac{4\kappa_1^{1/p}}{2^{-1}-2^{-d/2}} \le\frac{d^{1/2}}{6^{1/2}
\wedge[(1+p)(1+p/2)]^{1/p}} +28\cdot\kappa_1^{1/p}.
\]
In the case $p\le1$, estimate
(\ref{c-vs-hatc})
with $\varTheta(r)=\frac1{p+1}r^{p+1}$ yields
\[
\widehat{\mathfrak c}_\infty\le\widehat{\mathfrak c}_0+
\sum_{m=0}^\infty2^{-(m+1)(d/2+1)}\cdot
\frac1{p+1}2^{(m+1)(p+1)}= \widehat{\frak c}_0+\frac1{(p+1)
\bigl(2^{d/2-p}-1 \bigr)},
\]
provided $p<d/2$. Together with Proposition~\ref{2thm4} this yields the claim.
\end{pf}

%
%co5.21 #&#
\begin{corollary}
\textup{(i)} For all $p\in(0,\infty)$,
\[
\frac{1}{\sqrt{2\pi e}} \le\liminf_{d\to\infty}\frac{\mathfrak
c_\infty^{1/p}}{d^{1/2}} \le
\limsup_{d\to\infty}\frac{\mathfrak c_\infty^{1/p}}{d^{1/2}} \le \frac1{\sqrt{6}\wedge \bigl[(1+p)
(1+p/2) \bigr]^{1/p}}.
\]
Note that the ratio of right and left-hand sides is less than 5, and
for $p\le2$ even less than $2$.

\textup{(ii)} For all $p\in(0,\infty)$ there exist constants $k,k'$ such that
for all
$d>2(p\wedge1)$,
\[
k\cdot d^{p/2} \le\frak c_\infty\le k'\cdot
d^{p/2}.
\]
\end{corollary}

%s6 #&#
\section{Optimal semicouplings with bounded second marginal}\label{sapp}

The goal of this chapter is to prove Theorem~\ref{euQ+q} ($={}$Theorem
\ref{euQ+q2}), the crucial existence and uniqueness result for
optimal semicouplings between the Lebesgue measure and the point
process restricted to a bounded set.

Throughout this chapter, we fix the cost function $c(x,y)=\vartheta
(|x-y|)$ with
$\vartheta$---as before---being a strictly increasing, continuous
function from $\mathbb R_+$ to $\mathbb R_+$ with $\vartheta(0)=0$ and
$\lim_{r\to\infty}\vartheta(r)=\infty$. In dimension one we exclude
the case $\vartheta(r) = r.$

%
%le6.1 #&#
\begin{lemma}\label{leb-dirac} Suppose there is given a finite set
$\Xi
=\{\xi_1,\ldots,\xi_k\}\subset\R^d$ and a probability density
$\rho\in
L^1(\R^d,\leb)$.
\begin{longlist}[(iii)]
\item[(i)] There exists a unique coupling $q$ of $\rho\leb$ and $\sigma
=\frac
1k\sum_{\xi\in\Xi}\delta_\xi$ which minimizes the cost function
$\CCost
(\cdot)$.

\item[(ii)] There exists a ($\leb$-a.e. unique) map $T\dvtx\{\rho>0\}\to
\Xi$ with
$T_*(\rho\leb)=\sigma$ which minimizes
$\int c(x,T(x))\rho(x) \,d\leb(x)$.

\item[(iii)] There exists a ($\leb$-a.e. unique) map $T\dvtx\{\rho>0\}\to
\Xi$ with
$T_*(\rho\leb)=\sigma$ which is $c$-monotone (in the sense that the
closure of $\{(x,T(x))\dvtx\rho(x)>0\}$ is a
$c$-cyclically monotone set).

\item[(iv)] The minimizers in \textup{(i)}, \textup{(ii)} and \textup{(iii)} are related by
$q=(\mathrm{id},T)_*(\rho\leb)$ or, in other words,
\[
dq(x,y) = d\delta_{T(x)}(y) \rho(x) \,d\leb(x).
\]
\end{longlist}
\end{lemma}

\begin{pf}
We prove the lemma in three steps.
\begin{longlist}[(a)]
\item[(a)] By compactness of $\Pi(\rho\leb,\sigma)$ w.r.t. weak convergence
and continuity of $c(\cdot,\cdot)$, there is a coupling $q$ minimizing
the cost function $\CCost(\cdot)$; see also~\cite{villani2009optimal},
Theorem 4.1.

\item[(b)] Write $\rho\leb=\dvtx\lambda=\sum_{i=1}^k \lambda_i$ where
$\lambda_i(\cdot):=q(\cdot \times\{\xi_i\})$ for each $i=1,\ldots,k$. We
claim that the
measures $(\lambda_i)_i$ are mutually singular. Assuming that there is
a Borel set $N$ such that for some $i\neq j$ we have $\lambda_i(N)=\alpha>0$ and $\lambda_j(N)=\beta>0$, we will redistribute the
mass on N being transported to $\xi_i$ and $\xi_j$ in a cheaper way.
This will show that the measures $(\lambda_i)_i$ are mutually singular.
In particular, the proof implies the existence of a measurable
$c$-monotone map T such that $q=(\mathrm{id},T)_*(\rho\leb)$.

We may assume w.l.o.g. that $(\rho\leb)(N)=\alpha+\beta$. Otherwise
write $\rho=\rho_1+\rho_2$ such that on N $\d\lambda_i(x)+\d
\lambda_j(x)=\d(\rho_1\leb)(x)$, and just work with the density~$\rho_1$.

Put $f(x):=c(x,\xi_i)-c(x,\xi_j)$. As $c(\cdot,\cdot)$ is
continuous, f
is continuous. The function $c(x,y)$ is a strictly increasing function
of the distance $|x-y|$. Thus, the level sets $\{f\equiv b\}$ define
(locally) $(d-1)$ dimensional submanifolds (e.g., use implicit function
theorem for non smooth functions, see Corollary 10.52 in \cite
{villani2009optimal}) changing continuously with b. Choose $b_0$ such
that $\rho\leb(\{f< b_0\}\cap N)=\alpha$ [which implies $\rho\leb
(\{
f>b_0\}\cap N)=\beta$] and set $N_i:=\{f<b_0\}\cap N$ and $N_j:=\{
f\geq
b_0\}\cap N$.\vadjust{\goodbreak}

For $l=i,j$,
\[
d\tilde\lambda_l(x):= d\lambda_l(x)-1_N(x)
\,d \lambda_l(x)+1_{N_l}(x)\,d(\rho\leb) (x).
\]
For $l\neq i,j$ set $\tilde\lambda_l= \lambda_l$. By construction,
$\tilde q=\sum_{l=1}^k \tilde\lambda_l\otimes\delta_{\xi_l}$ is a
coupling of $\rho\leb$ and $\sigma$. Moreover, $\tilde q$ is
$c$-cyclically monotone on $N$, that is, $\forall x_i\in N_i,x_j\in
N_j$ we have
\[
c(x_i,\xi_i)+c(x_j,\xi_j) \leq
c(x_j,\xi_i)+c(x_i,\xi_j).
\]
Furthermore, the set where equality holds is a null set because
$c(x,y)$ is a strictly increasing function of the distance. Then we have
\begin{eqnarray*}
\CCost(q)-\CCost(\tilde q) &=& \int_{N}c(x,
\xi_i)\,\d\lambda_i(x) + c(x,\xi_j)\,\d
\lambda_j(x)
\\
&&{}- \int_{N_i}c(x,\xi_i)\,\d\tilde
\lambda_i(x)-\int_{N_j}c(x,\xi_j)\,\d
\tilde\lambda_j(x)> 0,
\end{eqnarray*}
by cyclical monotonicity.
This proves that $\lambda_i$ and $\lambda_j$ are singular to each other.

Hence, the family $(\lambda_i)_{i=1,\ldots,k}$ is mutually singular
which in turn implies that there exist Borel sets
$S_i\subset\R^d$ with $\dot{\bigcup}_{i}S_i=\R^d$ and $\lambda_i(S_j)=0$
for all $i\not=j$. Define the map $T\dvtx\R^d\to\Xi$ by $T(x):=\xi_i$ for
all $x\in S_i$. Then $q=(\mathrm{id},T)_*(\rho\leb)$.

\item[(c)] Assume there are two minimizers of the cost function $\CCost$, say
$q_1$ and $q_2$. Then $q_3:=\frac12 (q_1+q_2)$ is a minimizer as well.
By step (b) we have $q_i=(\mathrm{id},T_i)_*\rho\leb$ for $i=1,2,3$. This implies
\begin{eqnarray*}
d\delta_{T_3(x)}(y) \,d\rho\leb(x)&=& dq_3(x,y) = d \bigl(
\tfrac12q_1(x,y)+\tfrac12 q_2(x,y) \bigr)
\\
&=& d \bigl(\tfrac12\delta_{T_1(x)}(y)+\tfrac12\delta_{T_2(x)}(y)
\bigr) \,d\rho\leb(x).
\end{eqnarray*}
This, however, implies $T_1(x)=T_2(x)$ for $\rho\leb$ a.e. $x\in\R^d$
and thus $q_1=q_2$.\quad\qed
\end{longlist}
\noqed\end{pf}

%
%re6.2 #&#
\begin{remark}
(1) In dimension one we exclude the case $c(x,y)=|x-y|$ because
the optimal coupling between an absolutely continuous measure and a
discrete measure need not be unique. In higher dimensions it is unique,
as we get strict inequalities in the triangle inequalities. A
counterexample for one dimension is the following. Take $\lambda$ to be
the Lebesgue measure on $[0,1]$ and put $\mu=\frac13 \delta_0 +
\frac
23\delta_{1/16}.$ Then, for any $a\in[1/16,1/3]$,
\begin{eqnarray*}
q_a(dx,dy) &= & 1_{[0,a)}(x)\delta_0(dy)
\lambda(dx)
\\
&&{}+1_{[a,2/3 + a)}(x)\delta_{1/16}(dy)\lambda(dx) +
1_{[a+2/3,1]}(x)\delta_0(dy)\lambda(dx)
\end{eqnarray*}
is an optimal coupling of $\lambda$ and $\mu$ with $\CCost(q_a)= 11/24$.\vspace*{-6pt}
\begin{longlist}[(2)]
\item[(2)] In the case $\vartheta(r)=r^2$, there exists a convex
function $\varphi\dvtx\{\rho>0\}\to\R$ such that
\[
T(x)=\nabla\varphi(x)\qquad \mbox{for }\leb\mbox{-a.e. }x.\vadjust{\goodbreak}
\]
More generally, if $\vartheta(r)=r^p$ with $p>1$, then the map $T$ is
given as
$T(x) = x+|\nabla\psi(x)|^{{(2-p)}/{(p-1)}}\cdot\nabla\psi(x)$
for some $|\cdot|^p$-convex function $\psi\dvtx\{\rho>0\}\to\R$.
\end{longlist}

\end{remark}

%
%pr6.3 #&#
\begin{proposition}\label{uniqueq}
For each finite set $\Xi\subset\R^d$ there exists a unique semicoupling
$q$ of
$\leb$ and $\sigma=\sum_{\xi\in\Xi}\delta_{\xi}$ which
minimizes the
cost functional $\CCost(\cdot)$.
\end{proposition}
\begin{pf}
(i) The functional $\CCost(\cdot)$ on $\mathcal M(\R^d\times\R^d)$ is
lower semicontinuous w.r.t. weak topology. Indeed,
if $\eta_n\to\eta$ weakly, then with $c_k(x,y):=\min\{\vartheta
(|x-y|),k\}$
\[
\liminf_n \CCost(\eta_n)\ge\sup_k \biggl[
\lim_n \int c_k \,d\eta_n \biggr]=
\sup_k \int c_k \,d\eta=\CCost(\eta).
\]

(ii) Let $\frak Q$ denote the set of all semicouplings of $\leb$ and
$\sigma$ and $\frak Q_1$ the subset of those $q\in\frak Q$ which
satisfy $\frac12 \CCost(q)\le\inf_{q'\in\frak Q} \CCost(q')=:c$. Then
$\frak Q_1$ is relatively compact w.r.t. the weak topology. Indeed,
$q(\R^d\times\complement\Xi)=0$ for all $q\in\frak Q_1$ and
\[
q \bigl(\complement K_r(\Xi)\times\Xi \bigr)\le\frac1{\vartheta(r
)} \cdot\CCost(q) \le\frac2{\vartheta(r )} c
\]
for each $r>0$ where $K_r(\Xi)$ denotes the closed $r$-neighborhood of
$\Xi$ in $\R^d$. Thus for any $\varepsilon>0$ there exists a compact set
$K={K_r(\Xi)}\times\Xi$ in $\R^d\times\R^d$ such that
$q(\complement
K)\le\varepsilon$ uniformly in $q\in\frak Q_1$.

(iii) The set $\frak Q$ is closed w.r.t. weak convergence. Indeed, if
$q_n\to q$, then $(\pi_1)_*q_n\to(\pi_1)_*q$ and $(\pi_2)_*q_n\to
(\pi_2)_*q$.

Thus, $\frak Q_1$ is compact and $\CCost(\cdot)$ attains its minimum on
$\frak Q$ (or equivalently on $\frak Q_1$).

(iv) Now let a minimizer $q$ of $\CCost(\cdot)$ on $\frak Q$ be given, and
let $\lambda=(\pi_1)_*q$ denote its first marginal. Then $\lambda
=\rho
\cdot\leb$ for some density $0\le\rho\le1$ on $\R^d$. Our first claim
will be that \textit{$\rho$ only attains values \textup{0} and \textup{1}.}

Indeed, put $U=\{\rho>0\}$.
According to the previous Lemma~\ref{leb-dirac}, there exists an a.e.
unique ``transport map'' $T\dvtx U\to\Xi$ s.t.
\[
q=(\mathrm{id}, T)_*\lambda.
\]
For a given ``target point'' $\xi\in\Xi$, $U_\xi:=U\cap T^{-1}(\xi
)$ is
the set of points which under the map $T$ will be transported to the
point $\xi$. Within this set, the density $\rho$ has values between 0
and 1 and its integral is 1. If the density is not already equal to 1
we can replace it by another one which gives maximal mass to the points
which are closest to the target $\xi$.
Indeed, put $r(\xi):=\inf\{r>0 \dvtx\leb(K_r(\xi)\cap U_\xi)\geq
1\}$
and $\tilde\lambda:=\tilde\rho\cdot\leb$ with
\[
\tilde\rho(x)=1_{\bigcup_{\xi\in\Xi} K_{r(\xi)}(\xi)\cap U_\xi}(x).
\]
Then
\[
\tilde q:=(\mathrm{id}, T)_*\tilde\lambda
\]
defines a semicoupling of $\leb$ and $\sigma$ with $\CCost(\tilde
q)\le
\CCost(q)$. Moreover, it holds that $\CCost(\tilde q)=\CCost(q)$ if and
only if
$\tilde\rho=\rho$ a.e. on $\R^d$. The latter is equivalent to $\rho
\in
\{0,1\}$ a.e.

(v)
Assume there are two optimal semicouplings $q_1$ and $q_2$ whose first
marginals have density $1_{U_1}$ and $1_{U_2}$, respectively. Then
$q:=\frac{1}{2}(q_1+q_2)$ is optimal as well and its first marginal has
density $\frac{1}{2}(1_{U_1}+1_{U_2})$. By the previous part (iv) of
this proof the density can attain only values $0$ or $1$. Therefore, we
have $U_1=U_2$ (up to measure zero sets) and $q_1=q_2$.
\end{pf}

%
%le6.4 #&#
\begin{lemma} Given a bounded Borel set $A\subset\R^d$, let
$\mathcal M_{\mathrm{count}}(A)=\{\sigma\in\mathcal M_{\mathrm{count}}(\R^d)\dvtx
\sigma(\R^d\setminus A)=0\}$ denote the set of finite counting
measures which
are concentrated on $A$.
Define $\Upsilon\dvtx \mathcal M_{\mathrm{count}}(A)\to\mathcal M(\R^d\times
\R^d)$
the map which assigns to each $\sigma\in\mathcal M_{\mathrm{count}}(A)$ the
unique $q\in\Pi_{s}(\leb,\sigma)$ which minimizes the cost functional
$\CCost(\cdot)$. Then $\Upsilon$ is continuous (w.r.t. weak convergence on
the respective spaces).
\end{lemma}

\begin{pf}
(i) Take a sequence $(\sigma_n)_n\subset\mathcal M_{\mathrm{count}}(A)$
converging weakly to some $\sigma\in\mathcal M_{\mathrm{count}}(A)$. Put
$q_n:=\Upsilon(\sigma_n)$ for $n\in\N$ and $q=\Upsilon(\sigma)$. We
have to prove that $q_n\to q$.\vspace*{-6pt}
\begin{longlist}[(iii)]
\item[(ii)] The weak convergence $\sigma_n\to\sigma$ implies that finally
all the measures $\sigma_n$ have the same total mass as $\sigma$, say
$k$. Hence, for each sufficiently large $n\in\N$ there exist points
$x_1^n,\ldots, x_k^n$ and Borel sets $S_1^n,\ldots, S_k^n$ such that
\[
\sigma_n=\sum_{i=1}^k
\delta_{x_i^n},\qquad q_n=\sum_{i=1}^k
1_{S_i^n}\leb\otimes\delta_{x_i^n}.
\]
Similarly
$\sigma=\sum_{i=1}^k \delta_{x_i}$ and $q=\sum_{i=1}^k 1_{S_i}\leb
\otimes\delta_{x_i}$
with suitable points $x_1,\ldots, x_k$ and Borel sets $S_1,\ldots, S_k$.
Weak convergence moreover implies that for each $i=1,\ldots,k$,
\[
x_i^n\to x_i \qquad\mbox{as }n\to\infty.
\]

\item[(iii)] Based on the representations of $q$ and $\sigma_n$, we can
construct a semicoupling $\hat q_n$ of $\leb$ and $\sigma_n$ as follows:
\[
\hat q_n=\sum_{i=1}^k
1_{S_i}\leb\otimes\delta_{x_i^n}.
\]
Then by continuity of $\vartheta$ and dominated convergence theorem,
\begin{eqnarray*}
\limsup_n\CCost(\hat q_n)&=& \limsup_n\sum
_{i=1}^k \int_{S_i}
\vartheta \bigl(\bigl|y-x_i^n\bigr| \bigr)\,dy
\\
&=& \sum_{i=1}^k \int
_{S_i}\vartheta(|y-x_i|)\,dy = \CCost(q).
\end{eqnarray*}
And of course $\CCost(q_n)\le\CCost(\hat q_n)$. Thus
\[
\limsup_n\CCost(q_n)\le\CCost(q).
\]
\item[(iv)] The sequence $(q_n)_n$ is relatively compact in the weak
topology of $\mathcal M(\R^d\times\R^d)$. Therefore, there is a
subsequence, denoted again by $(q_n)_n$, converging weakly to some
measure $\tilde q\in\mathcal M(\R^d\times\R^d)$. It follows that
$(\pi_2)_*q_n\to(\pi_2)_*\tilde q$ and thus $(\pi_2)_*\tilde
q=\sigma$.
Similarly, $(\pi_1)_*\tilde q\le\leb$. Thus $\tilde q\in\Pi_{s}(\leb
,\sigma)$.
Lower semicontinuity of the cost functional implies
\[
\CCost(\tilde q)\le\liminf_{n\to\infty}\CCost(q_n).
\]

\item[(v)] Summarizing, we have proven that $\tilde q$ is a semicoupling of
$\leb$ and $\sigma$ with
\[
\CCost(\tilde q)\le\CCost(q).
\]
Since $q$ is the unique minimizer of the cost functional among all
these semicouplings, it follows that $\tilde q = q$. In other words,
\[
\lim_{n\to\infty}\Upsilon(\sigma_n)=\Upsilon \Bigl(
\lim_{n\to\infty}\sigma_n \Bigr).
\]
This proves the continuity of $\Upsilon$.\quad\qed
\end{longlist}
\noqed\end{pf}

For a given $\omega$ let us apply the previous results to the measure
\[
\sigma=1_A\mu^\omega=\sum_{\xi\in\Xi(\omega)\cap A}
\delta_\xi
\]
for a realization $\mu^\omega$ of the point process. Then, there is a
unique minimizer---in the sequel denoted by $q^\omega_{A}$---of the
cost functional $\CCost$ among all semicouplings of $\leb$ and
$1_A\mu^\omega$.

%
%le6.5 #&#
\begin{lemma} For each bounded Borel set $A\subset\R^d$ the map
$\omega
\to q_A^\omega$ is measurable.
\end{lemma}

\begin{pf}
We saw that the map $\Upsilon\dvtx\mathcal M_{\mathrm{count}}(A)\to\mathcal
M(\R^d\times\R^d)$, $\sigma\mapsto\Upsilon(\sigma)$ assigning to each
counting measure $\sigma$ its unique minimizer of $\CCost(\cdot)$ is
continuous. By definition of the point process, $\omega\mapsto\mu^\omega
$ is measurable. Hence, the map
\[
\omega\mapsto q_A^\omega=\Upsilon \biggl(\sum
_{\xi\in A\cap\Xi(\omega
)}\delta_{\xi} \biggr)
\]
is measurable.
\end{pf}

%
%th6.6 #&#
\begin{theorem}\label{euQ+q2}
\textup{(i)} For each bounded Borel set $A\subset\R^d$ there exists a unique
semicoupling $\Q_A$ of $\leb$ and $(1_A\mu^\bullet)\mathbb P$ which
minimizes the mean cost functional
$\Cost(\cdot)$.\vadjust{\goodbreak}%\vspace*{-6pt}
\begin{longlist}[(iii)]
\item[(ii)] $\Q_A$ can be disintegrated as $d\Q_A(x,y,\omega):=dq_A^\omega
(x,y) \,d{\mathbb P}(\omega)$ where for $\mathbb P$-a.e. $\omega$ the
measure $q_A^\omega$ is the unique minimizer of the cost functional
$\CCost(\cdot)$ among the semicouplings of $\leb$ and $1_A\mu^\omega$.

\item[(iii)] $\Cost(\Q_A)=\int_\Omega\CCost(q_A^\omega) \,d\mathbb
P(\omega).$
\end{longlist}
\end{theorem}

\begin{pf}
The existence of a minimizer is proven along the same lines as in the
previous proposition: We choose an approximating sequence
$\Q_n$ in $\mathcal M(\R^d\times\R^d\times\Omega)$---instead of a
sequence $q_n$ in $\mathcal M(\R^d\times\R^d)$---minimizing the lower semicontinuous functional $\Cost(\cdot)$.
Existence of a limit follows as before from tightness of the set of all
semicouplings $\Q$ with $\Cost(\Q)\le2\inf_{\tilde\Q}\Cost
(\tilde\Q)$.

For each semicoupling $\Q$ of $\leb$ and $\mu^\bullet\mathbb P$ with
disintegration as $q^\bullet\mathbb P$, we obviously have
\[
\Cost(\Q)=\int_\Omega\CCost \bigl(q^\omega \bigr)
\,d \mathbb P(\omega).
\]
Hence, $\Q$ is a minimizer of the functional $\Cost(\cdot)$ (among all
semicouplings of $\leb$ and $\mu^\bullet\mathbb P$) if and only if for
$\mathbb P$-a.e. $\omega\in\Omega$ the measure $q^\omega$ is a
minimizer of the functional $\CCost(\cdot)$ (among all semicouplings of
$\leb$ and $\mu^\omega$).

Uniqueness of the minimizer of $\CCost(\cdot)$ therefore implies uniqueness
of the minimizer of $\Cost(\cdot)$.
\end{pf}

%
%co6.7 #&#
\begin{corollary}
For each $z\in\R^d$ and each bounded Borel set $A\subset\R^d$, the
measure $\Q_A$ satisfies
\[
Q_A(B,C,\omega) = Q_{A+z}(B+z,C+z,\omega+z)
\]
for all Borel sets $B,C\in\mathcal B(\R^d).$
\end{corollary}

\begin{pf}
Since $\leb$ is equivariant and $\mu^\bullet$ is equivariant, the claim
follows from the uniqueness of the minimizer of the cost functional
$\Cost(\cdot)$.
\end{pf}

%
%re6.8 #&#
\begin{remark}
As before, for a finite set $\varXi\subset\R^d$ put $\sigma=\sum_{\xi\in
\varXi}\delta_\xi$. Let $q$ be a semicoupling of $\leb$ and $\sigma$.
Then $q$ minimizes $\CCost(\cdot)$ if and only if the support of $q$ is
$c$-cyclically monotone and $q$ is
\emph{$c$-sequentially monotone} in the following sense:
\[
\sum_{i=1}^n c(x_i,
\xi_i)\leq\sum_{i=1}^n
c(x_{i+1},\xi_i)
\]
for all $n\in\N, \{(x_i,\xi_i)\}_{i=1}^n\in\supp(q), \forall
x_{n+1}\notin\supp((\pi_1)_*q)$.
\end{remark}

\begin{pf}
Let $q$ be the unique minimizing semicoupling. The cyclical monotonicity
follows from the general theory of optimal transportation; cf. Section~\ref{sot}.
Put\vadjust{\goodbreak} $U:=\supp((\pi_1)_*q)$. Assume that $q$ is not
sequentially monotone. Then there are $n\in\N, x = x_{n+1}\in
\complement U, \{(x_i,\xi_i)\}_{i=1}^n\in\supp(q)$ such that
\[
\sum_{i=1}^n c(x_i,
\xi_i)> \sum_{i=1}^n
c(x_{i+1},\xi_i).
\]
By continuity of the cost function, there are (compact) neighborhoods
$U_i$ of $x_i$ and $V_i$ of $\xi_i$ such that $U_{n+1}\cap
U=\varnothing$ and
\[
\sum_{i=1}^nc(u_i,v_i)>
\sum_{i=1}^nc(u_{i+1},v_i),
\]
whenever $u_i\in U_i$ and $v_j\in V_j$. Moreover, as $\supp(\sigma)$ is
discrete, we can assume (by shrinking $V_j$ slightly if necessary) that
$V_j\cap\supp(\sigma)=\{\xi_j\}$. As $(x_i,\xi_i)\in\supp(q)$ for
$1\leq i\leq n$, we have $\inf_i q(U_i\times\{\xi_i\})>0.$ Set
$\lambda
:=\inf\{ q(U_1\times\{\xi_1\}),\ldots,q(U_n\times\{\xi_n\}), \leb
(U_{n+1})\}.$ Then we can reallocate mass to define a new measure with
less cost. Indeed, we can choose subsets $\tilde U_i\subset U_i, \tilde
U_i\times\{\xi\}_i\subset\supp(q)$ with $\leb(\tilde U_i)=\lambda
$ and
define a new measure $\tilde q$ by
\begin{eqnarray*}
&&d\tilde q(x,y)
\\
&&\qquad=dq(x,y)-\frac1 n \sum_{i=1}^n
1_{\tilde
U_i\times\{\xi_i\}}(x,y)\,d\leb(x) + \frac1 n \sum_{i=1}^n
1_{\tilde
U_{i+1}\times\{\xi_i\}}(x,y)\,d\leb(x).
\end{eqnarray*}
By assumption, we have $\CCost(\tilde q)<\CCost(q)$. Hence, $q$ is not
minimizing $\CCost$.

For the other direction let us assume that $q$ is cyclically monotone and
sequentially monotone but not minimizing $\CCost(\cdot)$. Then there is a
Borel set $\tilde U \neq U (=\supp((\pi_1)_*q))$ (by uniqueness of
optimal transportation of fixed measures) and a unique $\CCost$
minimizing coupling $\tilde q$ of $1_{\tilde U}\leb$ and $\sigma$ such
that $\CCost(\tilde q) \leq\CCost(q)$, and the support of $\tilde q$
is cyclically monotone. As $\tilde U \neq U$ there is some $z\in\tilde
U\setminus U$ which is transported by $\tilde q$ to $\xi_0$, say. For
$\xi\in\varXi$ set $S_\xi:=\{x\in\R^d \dvtx(x,\xi)\in\supp
(q)\}$ and
similarly $\tilde S_\xi$ for $\tilde q$. By sequential monotonicity of
q for all $x_0\in S_{\xi_0}$, we must have $c(x_0,\xi_0)\leq c(z,\xi_0)$. Moreover, the set $\{x\in S_{\xi_0}\dvtx c(x,\xi_0)=c(z,\xi_0)\}$ is
a $\leb$ null set. Thus there is a set $\hat S_{\xi_0}\subset S_{\xi
_0}$ of Lebesgue measure one such that for all $x\in\hat S_{\xi_0}$, we
have $c(x,\xi_0)<c(z,\xi_0)$. By the first part, we know that a
minimizing semicoupling is sequentially monotone. Thus $\hat S_{\xi
_0}\subset\tilde U$ and also $S_{\xi_0}\subset\tilde U$ (in
particular if $\varXi=\{\xi_0\}$ we are done).

Moreover, by assumption there is some $x_1\in S_{\xi_0}\setminus
\tilde
S_{\xi_0}$ which is transported by $\tilde q$ to some $\xi_1\in
\varXi$.
Then $S_{\xi_1}\setminus\tilde S_{\xi_1}$ is not empty. If $S_{\xi
_1}\cap\,\complement\tilde U\neq\varnothing$, we choose $x_2\in
S_{\xi
_1}\cap\,\complement\tilde U$ and stop. If $S_{\xi_1}\subset\tilde U$,
there is $x_2\in S_{\xi_1}\setminus\tilde S_{\xi_1}$ which is
transported by $\tilde q$ to some $\xi_2$. If $\xi_2\in\{\xi_0,\xi_1\}$
(i.e., $\xi_2=\xi_0$), we choose $x_2\in\tilde S_{\xi_2}\cap S_{\xi
_1}$ and stop. Otherwise we proceed in the same manner until either
$S_{\xi_k}\cap\,\complement\tilde U\neq\varnothing$ or $\xi_k\in\{
\xi_0,\ldots,\xi_{k-2}\}$. By this procedure, we construct a sequence
$x_0,\ldots,x_k$ such that $x_j\in\tilde S_{\xi_j}\cap S_{\xi_{j-1}}$
for $1\leq j\leq k-1$, $x_0\in\tilde S_{\xi_0}\setminus U$, and either\vadjust{\goodbreak}
$x_k\in S_{\xi_k}\setminus\tilde U$ or $x_k\in\tilde S_{\xi_k}\cap
S_{y_{k-1}}=\tilde S_{\xi_j}\cap S_{y_{k-1}}$ for some $0\leq j\leq
k-2.$ In the latter case, we have by cyclical monotonicity for $\tilde
q$ and $q$,
\[
\sum_{i=j}^k c(x_i,
\xi_i)\leq\sum_{i=j}^k
c(x_{i+1},\xi_i)\leq\sum_{i=j}^k
c(x_i,\xi_i),
\]
where $\xi_k=\xi_j$ and $x_{k+1}=x_j$. Hence we have equality
everywhere. However, we can move the $x_i$ slightly to get a
contradiction. Thus, we need to have $x_k\in S_{\xi_k}\setminus\tilde
U$. Then we have by the sequential monotonicity of $\tilde q$ and $q$
\[
\sum_{i=0}^{k-1}c(x_i,
\xi_i)\leq\sum_{i=0}^{k-1}c(x_{i+1},
\xi_i) \leq\sum_{i=0}^{k-1}c(x_i,
\xi_i).
\]
Hence we need to have equality and therefore a contradiction as before.
Hence $\tilde q=q$.
\end{pf}

\section*{Acknowledgments}
The first author would like to thank Alexander Holroyd for pointing out
the challenges of $p<d/2$ in dimensions $d\le2$.
Both authors would like to thank Matthias Erbar for the nice pictures.

% \bibliography{bib}
% \bibliographystyle{alpha}
%
% imsref loaded by akundreckaite, 2013-02-15 09:57:51
%


%suskaldyti doi

\printaddresses

\end{document}